\documentclass[11pt]{amsart} 

\usepackage{amssymb,amsmath} 
\usepackage{enumitem} 
\usepackage[all]{xy}

\begin{document}
  
\setenumerate{label=(\alph*), font=\normalfont} 
 
 \newtheorem{thm}{Theorem}[section] 
 \newtheorem{prop}[thm]{Proposition} 
 \newtheorem{lem}[thm]{Lemma} 
 \newtheorem{cor}[thm]{Corollary} 
 \newtheorem{conj}[thm]{Conjecture} 
\newtheorem{question}[thm]{Question} 
  
 \theoremstyle{definition} 
 \newtheorem{example}[thm]{Example} 
 \newtheorem{defn}[thm]{Definition} 
 \newtheorem{remark}[thm]{Remark} 
 
\numberwithin{equation}{section} 
 
\newcommand{\bbN}{{\mathbb{N}}} 
\newcommand{\bbZ}{{\mathbb{Z}}} 
\newcommand{\bbR}{{\mathbb{R}}} 
\newcommand{\bbP}{{\mathbb{P}}} 
\newcommand{\bbA}{{\mathbb{A}}} 
\newcommand{\bbG}{{\mathbb{G}}} 
\newcommand{\bbC}{{\mathbb{C}}} 
\newcommand{\bbF}{{\mathbb{F}}} 
\newcommand{\bbQ}{{\mathbb{Q}}} 
\newcommand{\Span}{\operatorname{span}} 
\newcommand{\Char}{\operatorname{char}} 
\newcommand{\Sp}{\operatorname{Sp}} 
\newcommand{\diag}{\operatorname{diag}} 
\newcommand{\GL}{\operatorname{GL}} 
\newcommand{\Mat}{\operatorname{M}} 
\newcommand{\SL}{\operatorname{SL}} 
\newcommand{\PGL}{\operatorname{PGL}} 
\newcommand{\PSL}{\operatorname{PSL}} 
\newcommand{\Orth}{\operatorname{O}} 
\newcommand{\Ind}{\operatorname{Ind}} 
\newcommand{\Stab}{\operatorname{Stab}} 
\newcommand{\Aut}{\operatorname{Aut}} 
\newcommand{\Hom}{\operatorname{Hom}} 
\newcommand{\End}{\operatorname{End}} 
\newcommand{\Res}{\operatorname{Res}} 
\newcommand{\Spec}{\operatorname{Spec}} 
\newcommand{\rank}{\operatorname{rank}} 
\newcommand{\pr}{\operatorname{pr}} 
\newcommand{\id}{\operatorname{id}} 
\newcommand{\ed}{\operatorname{ed}} 
\newcommand{\rdim}{\operatorname{rdim}} 
\newcommand{\adim}{\mathbf{a}} 
\newcommand{\ST}{\operatorname{ST}} 
\newcommand{\pmed}{\operatorname{pmed}} 
\newcommand{\Xred}{\operatorname{X_{\rm red}}}
\newcommand{\trdeg}{\operatorname{trdeg}} 
\newcommand{\Gal}{\operatorname{Gal}} 
\newcommand{\Sch}{\operatorname{Sch}} 
\newcommand{\im}{\operatorname{im}} 
\newcommand{\Id}{\operatorname{Id}} 
\newcommand{\Pic}{\operatorname{Pic}} 
\newcommand{\Ext}{\operatorname{Ext}} 
\newcommand{\Sym}{\operatorname{S}} 
\newcommand{\Alt}{\operatorname{A}} 
\newcommand{\calA}{\mathcal{A}} 
\newcommand{\calB}{\mathcal{B}} 
\newcommand{\calC}{\mathcal{C}} 
\newcommand{\calF}{\mathcal{F}} 
\newcommand{\calG}{\mathcal{G}} 
\newcommand{\calI}{\mathcal{I}} 
\newcommand{\calL}{\mathcal{L}} 
\newcommand{\calP}{\mathcal{P}} 
\newcommand{\calR}{\mathcal{R}} 
\newcommand{\calS}{\mathcal{S}} 
\newcommand{\calX}{\mathcal{X}} 
\newcommand{\calY}{\mathcal{Y}} 
\newcommand{\calT}{\mathcal{T}} 
\newcommand{\Fields}{\mathbf{Fields}} 
\newcommand{\Var}{\mathbf{Var}} 
\newcommand{\gVar}{\textrm{-}\mathbf{Var}} 
\newcommand{\biratarrow}{\stackrel{\sim}{\dasharrow}} 
\newcommand{\Gr}{\operatorname{Gr}}  
\newcommand{\Yfnote}[2]{#1^{#2}}
 
\title{Pseudo-reflection groups and essential dimension} 
 
\author{Alexander Duncan} 
\address{Alexander Duncan\newline 
Department of Mathematics, University of Michigan, 
Ann Arbor, MI 48109, USA}
\thanks{A. Duncan was partially supported by
National Science Foundation
RTG grants DMS 0838697 and DMS 0943832.}
 
\author{Zinovy Reichstein} 
\address{Zinovy Reichstein\newline 
Department of Mathematics, University of British Columbia, 
Vancouver, BC V6T 1Z2, Canada}
\thanks{Z. Reichstein was partially supported by 
National Sciences and Engineering Research Council of
Canada Discovery grant 250217-2012.}
 
 
\begin{abstract}
We give a simple formula for the essential dimension 
of a finite pseudo-reflection group at a prime $p$  
and determine the absolute essential dimension 
for most irreducible pseudo-reflection groups.  We also 
study the ``poor man's essential dimension" of an arbitrary
finite group, an intermediate notion between the absolute 
essential dimension and the essential dimension at a prime $p$.
\end{abstract} 

\maketitle

\section{Introduction}

Let $k$ be a field and $G$ be a finite group. We begin by recalling 
the definition of the essential dimension $\ed_k(G)$.

A $G$-variety is a $k$-variety $X$ with a $G$-action.
A $G$-variety $X$ is \emph{primitive} if 
$G$ acts transitively on the irreducible 
components of $X_{\overline{k}}$. Here $\overline{k}$ denotes
the algebraic closure of $k$.  A \emph{compression} 
is a dominant $G$-equivariant $k$-map $X \dasharrow Y$,
where $X$ and $Y$ are primitive faithful $G$-varieties defined over $k$.
The essential dimension of a primitive faithful $G$-variety $X$,
denoted by $\ed(X)$, is defined as the minimal dimension
of $Y$, where $X$ is fixed, $Y$ varies, and
the minimum is taken over all compressions $X \dasharrow Y$.
The essential dimension $\ed_k(G)$ of $G$ is the maximal value of $\ed(X)$
over all primitive faithful $G$-varieties $X$ defined over $k$. 
This maximal value is attained 
in the case where $X = V$ is a finite-dimensional $k$-vector space 
on which $G$ acts
via a faithful representation $G \hookrightarrow \GL(V)$.
We will denote this numerical invariant of $G$ by $\ed_k(G)$,
or simply $\ed(G)$ when the base field $k$ is clear.

The notion of essential dimension has classical origins, even though 
it was only formalized relatively recently~\cite{bur}. In particular, 
Felix Klein showed (using different terminology) that
$\ed_{\bbC}(\Sym_5)=2$ in his 1884 book~\cite{fklein}.
In Galois-theoretic language, $\ed_k(G)$ is the minimal integer 
$d \geq 0$ such that for every field $K/k$ and every G-Galois 
field extension $L/K$, one can write $L \simeq K[x]/(f(x))$, 
where at most $d$ of the coefficients of the polynomial $f(x) \in K[x]$ 
are algebraically independent over $k$. This number naturally comes 
up in the construction of so-called ``generic polynomials" for the group $G$
in inverse Galois theory; see~\cite[Chapter 8]{jly}.
Essential dimension can also be defined in a broader context
as a numerical invariant of more general algebraic objects. 
In this paper our focus will be solely on finite groups. 
For surveys of the broader theory, we refer an interested reader 
to~\cite{icm,whatis,merkurjev}. 

The essential dimension has turned out to 
be surprisingly difficult to compute for many finite groups. 
For example, the exact value of
$\ed_{\bbQ}(\bbZ/n \bbZ)$ is only known for 
a few small values of $n$. The relative version of essential dimension
at a prime integer $p$, denoted by $\ed(G; p)$, has proved to 
be more accessible. If $X$ is a primitive faithful $G$-variety,
$\ed(X; p)$ is defined as the minimum of $\dim(Y)$ over 
all primitive faithful $G$-varieties $Y$ which admit a $G$-equivariant 
correspondence $X \rightsquigarrow Y$ of degree prime to $p$.
The essential dimension $\ed(G; p)$ is, once again, defined as
the minimal value of $\ed(X; p)$.
Recall that a correspondence $X \rightsquigarrow Y$ of degree $1$ is
the same thing as a dominant rational map $X \dasharrow Y$. 
Thus $\ed(X; p) \leqslant
\ed(X)$ and $\ed(G; p) \leqslant \ed(G)$ for every prime $p$.
The best known lower bound for $\ed(G)$ is usually deduced from
this inequality.

The computation of $\ed(G; p)$ is greatly facilitated by a 
theorem of N.~A.~Karpenko and A.~S.~Merkurjev~\cite{km}, which
asserts that 
\begin{equation} \label{e.km}
\ed(G; p) = \ed(G_p) = \rdim(G_p)  \, . 
\end{equation}
Here $G_p$ is any Sylow $p$-subgroup of $G$, and
for a finite group $H$, $\rdim(H)$ denotes 
the minimal dimension of a faithful representation of $H$ defined over $k$,
and we assume that $\zeta_p \in k$, where $\zeta_p$ is 
a primitive $p$th root of unity.
Note that since $[k(\zeta_p):k]$ is prime to $p$,
$\ed_k(G; p) = \ed_{k(\zeta_p)}(G; p)$.

The case where $G = \Sym_n$ is the symmetric
group is of particular interest because it relates to classical 
questions in the theory of polynomials; see~\cite{bur,bur2}. 
Here the relative essential dimension is known exactly for every prime
$p$,
\begin{equation} \label{e.ed-S_n-at-p}
\ed(\Sym_n; p) = \left\lfloor \dfrac{n}{p} \right\rfloor; 
\end{equation} 
see~\cite[Corollary 4.2]{mr09}.
The absolute essential dimension 
$\ed(\Sym_n)$ is largely unknown.
In characteristic zero we know only that
\begin{equation} \label{e.strict}
\max_p \, \ed(\Sym_n; p) = 
\left\lfloor \dfrac{n}{2}  \right\rfloor \leqslant 
\left\lfloor \dfrac{n+1}{2} \right\rfloor \leqslant \ed(\Sym_n) 
\leqslant n - 3 
\end{equation}
for any $n \geqslant 6$; see~\cite{bur}, \cite{duncan} and \cite{mm}.
We know even less about $\ed(\Sym_n)$ in prime characteristic.
 
The symmetric groups $\Sym_n$ belong to the larger family of pseudo-reflection
groups. Pseudo-reflection groups play an important role in representation 
theory and invariant theory of finite groups;
see, e.g.~\cite{kane,LTbook,SheTod54Finite}.
It is thus natural to try to compute $\ed(G)$ and $\ed(G; p)$, where
$G$ is a finite pseudo-reflection group, and $p$ is a prime.  The first steps
in this direction were taken by M.~MacDonald~\cite[Section 5.1]{mm}, who
computed $\ed(G; p)$ for all primes $p$ and all irreducible Weyl groups
$G$.  He also computed $\ed(G)$ for every irreducible Weyl group $G$,
except for $G = \Sym_n$ and $G = W({\bf E_6})$, 
the Weyl group of the root system of type ${\bf E_6}$.  
His proofs are based on case-by-case analysis.

The aim of this paper is twofold. First, we will generalize MacDonald's
results to all finite pseudo-reflection groups except the symmetric groups,
with a more uniform statement and proof.
Second, we will investigate a new intermediate notion between
$\max_p \, \ed(G; p)$ and $\ed(G)$, which we call ``poor man's 
essential dimension.''  

Throughout this paper we will
assume that $\Char(k)$ does not divide the order of $G$.
Our finite groups will be viewed as split algebraic groups over $k$.  
We will denote by $\bar{k}$ the algebraic closure of $k$ and by $\zeta_d$ 
a primitive $d$th root of unity in $\bar{k}$ where $d$ is a positive 
integer coprime to $\Char(k)$.  By a \emph{variety} we will mean 
a separated reduced scheme of finite type over $k$, not necessarily 
irreducible.  We will also adopt the following notational conventions
inspired by~\cite{Spr74Regular}.
Let $\phi \colon G \hookrightarrow \GL(V)$ be
a faithful representation of $G$ and $m$ be a positive integer
prime to the characteristic of $k$.
Set $V(g, \zeta_m) := \ker(\zeta_m I - \phi(g))$ to be
the $\zeta_m$-eigenspace of $g$ and let
\[
a_{\phi}(m) := \max_{g \in G} \, \dim V(g,\zeta_m) \, .
\]
Note that $V(g, \zeta_m)$ is defined over $k(\zeta_m)$ but may 
not be defined over $k$.
Replacing $g$ by a suitable power, we see that
$a_{\phi}(m)$ depends only on $\phi$ and $m$ and not 
on the choice of the primitive $m$th root of unity $\zeta_m$.
If the reference to $\phi$ is clear from the context, we will 
write $g$ in place of $\phi(g)$ and $a(m)$ in place of $a_{\phi}(m)$.
By convention, we set $a(m)=0$ if $m$ is a multiple of the
characteristic of $k$.

Recall that an element $g \in \GL(V)$ is a \emph{pseudo-reflection} if
it is conjugate to a diagonal matrix of the form 
$\diag(1, \dots, 1, \zeta)$, where $\zeta \ne 1$ is a root of unity.

\begin{thm} \label{thm.horizontal1} 
Let $G$ be a finite subgroup of $\GL(V)$.  Assume that the characteristic 
of the base field $k$ does not divide $|G|$.  Then

\smallskip 
(a) $\ed(G; p) \leqslant a(p)$ for every prime $p$. 

\smallskip
(b) Moreover, if $G$ is generated by pseudo-reflections then
$\ed(G; p) = a(p)$ for every prime $p$. 
\end{thm}

Suppose that $\phi \colon G \hookrightarrow \GL(V)$ is generated 
by pseudo-reflections with $n = \dim(V)$. 
Then $k[V]^G = k[f_1, \dots, f_n]$ for some homogeneous 
polynomials $f_1, \dots, f_n$.  Set $d_i := \deg(f_i)$.
The integers $d_1  \dots, d_n$ are called the {\em degrees of the fundamental 
invariants} of $\phi$. These numbers are uniquely determined by $\phi$ 
up to reordering. They are independent of the choice of $f_1, \ldots, f_n$
and can be recovered directly from the Poincar\'e series of $k[V]^G$;
see, e.g., \cite{kane} or \cite{LTbook}.
T.~A.~Springer~\cite[Theorem 3.4(i)]{Spr74Regular} showed that 
\begin{equation} \label{e.springer}
a(m) = | \{ i \, | \,  \text{$d_i$ is  
divisible by $m$} \} | \, .
\end{equation}
Note that while the base field $k$ is assumed to be the field 
of complex numbers $\bbC$
in~\cite[Theorem 3.4(i)]{Spr74Regular}, the above formula 
remains valid under our less restrictive assumptions on $k$; 
see, e.g.,~\cite[Section 33-1]{kane}.

Complex groups generated by pseudo-reflections have been classified
by  G.~C.~Shephard and J.~A.~Todd~\cite{SheTod54Finite}. 
Their classification lists $d_1, \dots, d_n$ in every case; 
Springer's theorem~\eqref{e.springer} makes it possible to 
read $a(m)$ directly off their table for every $G$ and every $m$.
The same can be done for other base fields $k$, as long  
as $\Char(k)$ does not divide $|G|$; for details and further references, 
see Section~\ref{sect.horizontal1-3}.

\begin{example}
For $G = W({\bf E_8})$ (group number $37$ in 
the Shephard-Todd classification),
the values of $d_1, \dots, d_8$ are \[ \text{$2$, $8$, $12$, $14$, $18$, 
$20$, $24$ and $30$,} \] 
respectively; see, e.g.,~\cite[Appendix D]{LTbook}. 
Counting how many of these numbers are divisible by each prime $p$ 
and applying Theorem~\ref{thm.horizontal1}(b) in combination 
with~\eqref{e.springer},
we recover the following values from~\cite[Table IV]{mm}:

\medskip
\begin{center}
\begin{tabular}{ | c | c | c | c | c | c | }
\hline
p & 2 & 3 & 5 & 7 & $> 7$ \\ 
\hline 
$ \; \; \ed(W({\bf E_8}); p) \; \; $ & $\quad 8 \quad$ & $\quad 4 \quad$ & 
$\quad 2 \quad$ & $\quad 1 \quad$ & $\quad 0 \quad$  \\
\hline
\end{tabular} \, . 
\end{center} 
\end{example}

Our proof of Theorem~\ref{thm.horizontal1} relies on both the
uniform arguments in Section~\ref{sect.horizontal1-1} 
and~\ref{sect.horizontal1-2} and some case-by case analysis
using the Shephard-Todd classification in~Section~\ref{sect.horizontal1-3}.

\medskip
Our next result, Theorem~\ref{thm.main3}, gives the exact value for
the \emph{absolute} essential dimension of all 
irreducible pseudo-reflection groups, except for $\Sym_n$. Recall that,
in the Shephard-Todd classification there are three infinite families:
the symmetric groups, the family $G(m,l,n)$
depending on three integer parameters ($m$, $l$, $n$),
and the cyclic groups.
In addition, there are 34 exceptional groups.

\begin{thm}\label{thm.main3}
Let $G \subset \GL(V)$ be an irreducible representation of
a finite group generated by pseudo-reflections.
Suppose $G$ is not isomorphic to a symmetric group $\Sym_n$
and $\Char(k)$ does not divide $|G|$.
Then  

\smallskip
(a) $\ed(G)= \dim(V) - 2 = 4$, if $G \simeq W({\bf E_6})$,

\smallskip
(b) $\ed(G)= \dim(V)-1 = n - 1$, if $G \simeq G(m,m,n)$ with
$m$, $n$ relatively prime,

\smallskip
(c) $\ed(G) = \dim(V)$ in all other cases.
\end{thm}

As we mentioned above, the exact value of $\ed(\Sym_n)$ is not known;
see~\eqref{e.strict}.  Part (a) answers an open question posed
in~\cite[Remark 5.2]{mm}.  The proof of this part relies 
on a geometric construction suggested to us by I.~Dolgachev.

We now recall that $\ed(G)$ is the minimal dimension of a versal $G$-variety 
and $\ed(G; p)$ is the minimal dimension of a $p$-versal
$G$-variety; see~\cite[Section 5]{gms}~and~\cite[Remark 2.5]{versal}.
\emph{Poor man's essential dimension}, denoted $\pmed(G)$, 
is defined as the minimal dimension of a $G$-variety which is 
simultaneously $p$-versal for every prime $p$. We have 
\begin{equation} \label{e.obvious}
\max_p \, \ed(G; p) \leqslant \pmed(G) \leqslant \ed(G) \, . 
\end{equation}
The term ``poor man's essential dimension'' is meant to suggest that
$\pmed(G)$ is a more accessible substitute for $\ed(G)$. Where exactly 
it fits between $\max_p \, \ed(G; p)$ and $\ed(G)$, is a key 
motivating question for this paper.

\begin{thm} \label{thm.horizontal2} 
Let $G$ be a finite subgroup of $\GL(V)$. Assume that 
$\Char(k)$ does not divide $|G|$.  Then

\smallskip 
(a) $\pmed(G) \leqslant \max_p \, a(p)$.

\smallskip
(b) Moreover, if $G$ is generated by pseudo-reflections then
\[ \pmed(G) = \max_p \, a(p) = \max_p \, \ed(G; p) \ . \]

\smallskip
In both parts the maximum is taken over all prime integers $p$.
\end{thm}

In particular,  $\pmed(\Sym_n) = \left\lfloor \dfrac{n}{2} \right\rfloor$ 
for every $n$, assuming $\Char(k) = 0$, a result we found 
somewhat surprising, considering that $\ed(\Sym_n) > 
\left\lfloor \dfrac{n}{2} \right\rfloor$ for every odd $n \ge 7$; 
see~\eqref{e.strict}. 

Our proof of Theorem~\ref{thm.horizontal2} 
relies on a variant of Bertini's Theorem; see Theorem~\ref{thm.bertini}.
If $k$ is an infinite field, Theorem~\ref{thm.bertini} is classical.
If $k$ is a finite field, we make use of the probabilistic 
versions of Bertini's smoothness and
irreducibility theorems, due to B.~Poonen~\cite{poonen1,poonen2} and 
F.~Charles and B.~Poonen~\cite{CPbertini}, respectively.
Note that~\cite{CPbertini} was motivated, in part, 
by the application in this paper.

In view of Theorem~\ref{thm.horizontal2}(b), it is natural to ask if 
\begin{equation} \label{e.elephant}
\pmed(G) = \max_p \, \ed(G; p)  
\end{equation}
for every finite group $G$.  In addition to the case of pseudo-reflection 
groups covered by Theorem~\ref{thm.horizontal2}(b), we will also prove that 
this is the case for alternating groups (Example~\ref{ex.Alt_n}) 
and for groups all of whose Sylow subgroups are abelian
(Proposition~\ref{prop.Sylow-abelian}).
A conjectural approach to proving~\eqref{e.elephant} 
for other finite groups is outlined at the end of
Section~\ref{sec:Agroups}.

\section{Proof of Theorem~\ref{thm.horizontal1}(a)}
\label{sect.horizontal1-1}

Throughout this section we fix a prime $p$ and assume that
the base field $k$ is of characteristic $\ne p$. 

\begin{lem} \label{lem.main1}
Let $V$ be a finite-dimensional $k$-vector space, 
and $G_p \subset \GL(V)$ be a finite $p$-group.
Assume $\zeta_p \in k$ and $V'$ is
a minimal (with respect to inclusion) faithful 
$G_p$-subrepresentation of $V$.
Then there exists a central element $g \in G_p$ of order $p$ 
such that $V' \subset V(g, \zeta_p)$, where $\zeta_p$ is
a primitive $p$th root of unity. 
\end{lem}

\begin{proof}
Let $C$ be the socle of $G_p$;  
i.e., the $p$-torsion subgroup of the centre $Z(G_p)$.

Decompose $V' = V_1 \oplus \cdots \oplus V_r$ as a direct sum of
irreducible $G_p$-representations.  Each $V_i$ decomposes into 
a direct sum of character spaces for $C$.  Since
$C$ is central, each of these character spaces is $G_p$-invariant.
As $V_i$ is irreducible as a $G_p$-module, there is
only one such component.  That is, $C$ acts on each $V_i$ by 
scalar multiplication via a character $\chi_i: C \to k^*$.

We will view the characters $\chi_i$ as elements of the dual 
group $C^* = \Hom(C, k^*)$. Note that
since $C$ is an elementary abelian $p$-group, $C^*$ has the natural 
structure of an $\bbF_p$-vector space.
Since $V'$ is minimal, an easy argument shows that
$\chi_1, \dots, \chi_r$ form an $\bbF_p$-basis of $C^*$;
see~\cite[Lemma 2.3]{mr10}.  Consequently, there is a unique element 
$g \in C$ such that $\chi_i(g)=\zeta_p$ for every $i=1, \ldots, r$.  
In other words, $V' \subset V(g, \zeta_p)$, as desired.
\end{proof}

\begin{proof}[Proof of Theorem~\ref{thm.horizontal1}(a)]
Neither $\ed(G; p)$ 
nor $a(p)$ will change if we replace $k$ by $k(\zeta_p)$. Hence, we
may assume without loss of generality that $k$ contains $\zeta_p$.
Let $G_p$ be a Sylow $p$-subgroup of $G$ and define
$V'$ and $g$ as in Lemma~\ref{lem.main1}. Then 
$V' \subset V(g, \zeta_p)$. Thus
\[ \ed(G; p) = \ed(G_p; p) \leqslant \ed(G_p) \leqslant \dim(V') \leqslant 
\dim \, V(g, \zeta_p) \leqslant a(p) \, , \]
as desired. Note that the inequality
$\ed(G_p) \leqslant \dim(V')$ is a consequence of the definition of
essential dimension; see, e.g.,~\cite[(2.3)]{icm}.
\end{proof}

We conclude this section with a refinement of Lemma~\ref{lem.main1} 
which will be used in the proofs of both Theorem~\ref{thm.horizontal1}(b) 
and Corollary~\ref{cor.main4}.

\begin{lem} \label{lem:pGroupMaximal}
Let $V$ be a finite-dimensional $k$-vector space, 
$G \subset \GL(V)$ be a finite group generated by 
pseudo-reflections, and $G_p$ be a $p$-Sylow subgroup of $G$.
Assume that $\zeta_p \in k$ and $V'$, $g$ are as in 
statement of Lemma~\ref{lem.main1}.  Then 
$\dim V(g,\zeta_p) = a(p)$. 
\end{lem}

\begin{proof}
By a theorem of Springer~\cite[Theorem 3.4(ii)]{Spr74Regular}
there exists an $h \in G$ such that $\dim V(h, \zeta_p) = a(p)$ and 
$V' \subset V(g, \zeta_p) \subset V(h,\zeta_p)$; 
see~\cite[Theorem 3.4(ii)]{Spr74Regular}. Springer originally 
proved this result over $\bbC$; a proof over an arbitrary 
base field (containing $\zeta_p$) can be found in~\cite[Chapter 33]{kane}.

After replacing $h$ by a suitable power, we may assume that the order 
of $h$ is a power of $p$.
Let $N = \{ x \in G \, | \, x(V') = V' \} $ be the stabilizer of $V'$ in $G$.
Note that $G_p \subset N$ and thus $G_p$ is a $p$-Sylow subgroup of $N$.
Since $V' \subset V(h,\zeta_p)$, we clearly have 
$h \in N$. On the other hand, since the order of $h$ is a power of $p$,
there exists an element $n \in N$ such that $h' = nhn^{-1}$ is in $G_p$.  
Note that $h$ acts on $V'$ as $\zeta_p \id_{V'}$, and hence, so does $h'$.
Now, $h'$ and $g$ both lie in $G_p$ and have identical actions on $V'$, 
which is a faithful representation of $G_p$.  Thus $h'=g$, and
$a(p) = \dim V(h, \zeta_p) = \dim V(h', \zeta_p) = \dim V(g, \zeta_p)$,
as desired. 
\end{proof}

\section{Proof of Theorem~\ref{thm.horizontal1}(b): First reductions}
\label{sect.horizontal1-2}

We now turn to the proof of Theorem~\ref{thm.horizontal1}(b).
In view of part (a), it suffices to show that
$\ed(G; p) \geqslant a(p)$. Since $\ed_k(G; p) \geqslant \ed_l(G; p)$, for
any field extension $l/k$, we may assume without loss of 
generality that $k$ is algebraically 
closed, and, in particular, that $\zeta_p \in k$.

Our proof of Theorem~\ref{thm.horizontal1}(b) will proceed by contradiction.  
We begin by studying a minimal counterexample, with the ultimate goal 
of showing that it cannot exist.

\begin{prop} \label{prop.minimal} 
Let $\phi \colon G \hookrightarrow \GL(V)$ be a counterexample
to Theorem~\ref{thm.horizontal1}(b) of minimal dimension. That is,
$V$ is a vector space of minimal dimension with the following properties:
there exists a finite group $G$, a faithful representation
$\phi \colon G \hookrightarrow \GL(V)$, and a prime $p$, such that
$\phi(G)$ is generated by pseudo-reflections, and 
\begin{equation} \label{e.minimal}
\ed(G; p) < a_{\phi}(p) \, . 
\end{equation}
Then
\begin{enumerate}
\item[(a)] $\dim(V) \geqslant 2$.

\item[(b)] $\phi$ is irreducible.

\item[(c)] Some element $g \in G$ of order $p$ acts on $V$ as a scalar. 
In particular, $a_{\phi}(p) = \dim(V)$.

\item[(d)] $G$ contains no elements of order $p$ with exactly two eigenvalues.

\item[(e)] $G$ contains no pseudo-reflections of order $p$.

\item[(f)] If $p = 2$ then $g = - \id_V$ is the unique element of order $2$ in $G$. 

\item[(g)] $G_p$ is contained in the commutator subgroup $[G, G]$.
Here, as usual, $G_p$ denotes a $p$-Sylow subgroup of $G$. 

\item[(h)] Let $g \in G$ be as in part (c) and
$\phi' \colon G \to \GL(V')$ be an irreducible representation
such that $\phi'(g) \ne 1$.  Then $\dim(V')$ is a multiple of $p$.
In particular, $\dim(V)$ is a multiple of $p$.

\item[(i)] $\dim(V) \geqslant 2p$.
\end{enumerate}
\end{prop}

\begin{proof} (a) Assume the contrary: $\dim(V) = 1$. 
In this case $G$ is a cyclic group.  If $|G|$ is divisible by $p$ 
then $\ed(G; p) = a(p) = 1$; otherwise  $\ed(G; p) = a(p) = 0$.
In both cases,~\eqref{e.minimal} fails, a contradiction. 

\smallskip
(b) Assume the contrary: $V = V_1 \oplus V_2$, where $V_1$ and $V_2$ are
proper $G$-stable subspaces.
Each pseudo-reflection $g \in G$ acts non-trivially
on exactly one summand $V_i$. For $i = 1, 2$, let
$G_i$ be the subgroup of $G$ generated by those reflections that 
act non-trivially on $V_i$. Then
$G$ is isomorphic to the direct product $G_1 \times G_2$,
and $\phi = \phi_1 \oplus \phi_2$, where $\phi$ restricts to 
$\phi_i \colon G_i \to \GL(V_i)$, and $\phi_1(G_1)$, $\phi_2(G_2)$ 
are generated by pseudo-reflections.
Note that $a_{\phi}(p) = a_{\phi_1}(p) + a_{\phi_2}(p)$.
In addition, by~\cite[Theorem 5.1]{km}, 
\[ \ed(G; p) = \ed(G_1; p) + \ed(G_2; p) \, . \]
By minimality of $\phi$, we have that
$\ed(G_1; p) \geqslant a_{\phi_1}(p)$ and
$\ed(G_2; p) \geqslant a_{\phi_2}(p)$.
Thus $\ed(G; p) \geqslant a_{\phi}(p)$, a contradiction.
 
\smallskip
(c) Choose $V'$ and $g$ as in Lemmas~\ref{lem.main1} 
and~\ref{lem:pGroupMaximal}. Recall that
$g$ is a central element of $G_p$ of order $p$ and
$a_{\phi}(p) = \dim \, V(g,\zeta_p)$. Set $W:= V(g, \zeta_p)$.
The element $g$ acts on $W$ as a scalar; our 
goal is to show that $W = V$. 

Let $S = \{ s \in G \, | \,  s (W) = W \}$
be the stabilizer of $W$ in $G$ and let $S_0$ 
be the subgroup of $S$ consisting of elements that fix $W$
pointwise. Note that since $g$ is central in $G_p$, we have $G_p \subset S$. 
Moreover, since $G_p$ acts faithfully on $V' \subset W$, we have
$G_p \cap S_0 = \{ 1 \}$.  Restricting the action of $S$ to $W$, we obtain
a faithful representation of $H = S/S_0$ on $W$, which we 
will denote by $\psi$.   By~\cite[Theorem 1.1]{LehMic03Invariant}, 
$\psi(H) \subset \GL(W)$ is generated by pseudo-reflections.
(Note that, while~\cite[Theorem 1.1]{LehMic03Invariant} assumes $k = \bbC$, 
its proof goes through under our less restrictive assumptions on $k$.)
By our construction,
\[ a_{\phi}(p) = \dim (W) = a_{\psi}(p) \, . \]
Since $G_p \subset S$ and $G_p \cap S_0 = \{ 1 \}$,
the quotient $H = S/S_0$ contains an isomorphic image of $G_p$, which
is a Sylow $p$-subgroup of $H$, so that
\[ \ed(G; p) = \ed(G_p; p) = \ed(H; p) \, . \]
Thus by~\eqref{e.minimal},
$\ed(H; p) = \ed(G; p) < a_{\phi}(p) = a_{\psi}(p)$. 
By the minimality of $\phi$, we see that $\dim(V) = \dim(W)$, 
i.e., $V = W = V(g, \zeta_p)$. This proves part (c). 

(d) Assume the contrary: an element $h$ of $G$ of order $p$
has exactly two distinct
eigenvalues, $\zeta_p^i$ and $\zeta_p^j$. After replacing $h$ by 
a suitable power of $h g^{-i}$, where $g$ is the central element 
we constructed in part (c),
we may assume that $i = 0$ and $j = 1$. Then
$V$ is the direct sum of eigenspaces $V_0 \oplus V_1$, where
$V_i = V(h, \zeta_p^i)$.  Let $G_1$ (resp. $G_0$)
be the subgroup of $G$ consisting of elements which fix
$V_0$ (resp. $V_1$) pointwise (note the reversed indices).

Since $G$ has order prime to the characteristic of $k$,
the direct sum $V_0 \oplus V_1$ is the unique decomposition of $V$ into
isotypic components for the group $\langle g, h \rangle$.  Since
$g h^{-1} \in G_0$ acts non-trivially
on $V_0$, the space $V_0$ is the unique $G_0$-invariant complement to
$V_1 = V^{G_0}$.  
Similarly, $V_1$ is the unique $G_1$-invariant complement to 
$V_0 = V^{G_1}$.
We now see that $G_0$ and $G_1$ commute and
$G_0 \cap G_1 = \{1 \}$. Hence, $G_0$ and $G_1$ generate a subgroup of $G$
isomorphic to $G_0 \times G_1$. By abuse of notation we shall
denote this group by $G_0 \times G_1$.

Note that $\phi$ restricts to faithful representations
$\phi_0 \colon G_0 \to \GL(V_0)$ and  $\phi_1 \colon G_1 \to \GL(V_1)$.
Since $\phi_0(g h^{-1}) = \zeta_p \id_{V_0}$ and 
$\phi_1(h) = \zeta_p \id_{V_1}$, we have
\[ \text{$a_{\phi_0}(p) = \dim(V_0)$ and $a_{\phi_1}(p) = \dim(V_1)$.} \] 
We now recall that by a theorem of
R.~Steinberg~\cite[Theorem 1.5]{Ste64Differential}, $G_0$ and $G_1 \subset 
\GL(V)$ are both generated by pseudo-reflections.
(In positive characteristic this is due to J.-P.~Serre~\cite{serre68};
cf.~\cite[Proposition 3.7.8]{derksen-kemper}.)
Since $G_1$ acts trivially on $V_0$ and $G_0$ acts trivially on $V_1$,
we conclude that $\phi_0(G_0)$ and $\phi_1(G_1)$ are 
also generated by pseudo-reflections. 

By the minimality of $\phi$, Theorem~\ref{thm.horizontal1}(b) holds 
for $\phi_0$ and $\phi_1$. Thus
\begin{gather*} 
 \ed(G; p) \geqslant \ed(G_0 \times G_1; p) = \ed(G_0; p) + \ed(G_1; p) = \\
a_{\phi_0}(p) + a_{\phi_1}(p) = \dim(V_0) + \dim(V_1) = \dim(V) =
a_{\phi}(p). 
\end{gather*}
Here the first equality is~\cite[Theorem 5.1]{km}, and the second  
follows from the minimality of $\phi$. The resulting inequality 
contradicts~\eqref{e.minimal}.

\smallskip
(e) By part (a), $\dim(V) \geqslant 2$. Hence, a pseudo-reflection 
has exactly two distinct eigenvalues, and (e) follows from (d).

\smallskip
(f) Every element of $\GL(V)$ of order $2$, other than $-\id_V$, has
exactly two distinct eigenvalues and thus cannot lie in $G$ by (d). 

\smallskip
(g) By (e), $G$ does not have any pseudo-reflections 
of order $p$, and hence of any order divisible by $p$. 
The finite abelian group $G/[G, G]$ is generated by 
the images of the pseudo-reflections.  All of these images 
have order prime to $p$. Hence, the order of $G/[G, G]$ 
is prime to $p$.  We conclude that $G_p \subset [G, G]$.

\smallskip
(h) Since $g$ is central, $\phi'(g) = \lambda \id_{V'}$, where
$\lambda$ is a primitive $p$th root of unity.  
Thus $\det \, \phi'(g) = \lambda^{\dim(V')}$.
On the other hand,  
by part (g), $g \in G_p \subset [G, G]$ and hence, 
$\det \phi'(g) = 1$.  Thus $\dim(V')$ 
is divisible by $p$.

\smallskip
(i) Let $C = \langle g \rangle$, where $g$ is as in part (c).
Applying \cite[Theorem 4.1]{icm}
(with $r = 1$) to the central exact sequence $1 \to C \to G \to G/C \to 1$
we obtain the inequality
\begin{equation} \label{e.ed}
\ed(G; p) \geqslant \gcd_{\phi'} \, \dim(\phi') \, , 
\end{equation} 
where $\phi' \colon G \to \GL(V')$ runs over all irreducible 
representations of $G$ such that the restriction of $\phi'$ to $C$
is non-trivial, or equivalently, $\phi'(g) \neq 1$. Note that
the statement of~\cite[Theorem 4.1]{icm} only gives this 
inequality for $\ed(G)$. However, it remains valid for $\ed(G; p)$; 
see \cite[Section 5]{icm} or the proof of \cite[Theorem 3.1]{lmmr}.

By part (h), $\dim(\phi')$ is divisible by $p$ for every such $\phi'$.
Thus $\ed(G; p) \geqslant p$. Assumption~\eqref{e.minimal} now tells us that 
$\dim(V) > p$. Since $\dim(V)$ is divisible by $p$ by (h), we 
conclude that $\dim(V) \geqslant 2p$. 
\end{proof}

\section{Conclusion of the proof of Theorem~\ref{thm.horizontal1}(b)}
\label{sect.horizontal1-3}

The remainder of the proof of Theorem~\ref{thm.horizontal1}(b) relies on
the classification of irreducible pseudo-reflection groups
due to Shephard and Todd~\cite{SheTod54Finite}.
Their classification consists of
three infinite families and 34 exceptional groups.
The first family contains the natural $(n-1)$-dimensional
representations of the group $S_n$.  The second family consists
of certain semidirect products of an abelian group and symmetric group.
The third family are simply the $1$-dimensional representations of
cyclic groups.  The representations of the exceptional groups
range from dimension $2$ through $8$.
We will denote the infinite families by $\ST_1$, $\ST_2$ and $\ST_3$,
and the exceptional groups $\ST_4$ through $\ST_{37}$,
following the numbering in~\cite{SheTod54Finite}.

Shephard and Todd worked 
over the field $k = \bbC$ of complex numbers. We are working over
a base field $k$ such that $\Char(k)$ does not divide $|G|$. 
As we explained at the beginning of the previous section, we 
may (and will) assume that $k$ is algebraically closed. 
Before we proceed with the proof of Theorem~\ref{thm.horizontal1}(b),
we would like to explain how the Shephard-Todd classification applies
in this more general situation. 

If $k$ is an algebraically closed field of characteristic 
zero, then any representation of a finite group over $k$ descends 
to $\overline{\bbQ} \subset k$; see~\cite[Section 12.3]{serre77}.  
Hence, this representation is defined over $\mathbb C$, and 
the entire Shephard-Todd classification remains valid over $k$. 

Now suppose $k$ is an algebraically closed field of positive 
characteristic.
Let $A = W(k)$ be its Witt ring. Recall that $A$ is a complete 
discrete valuation ring of characteristic zero,
whose residue field is $k$. Denote the fraction field of $A$
by $K$ and the maximal ideal by $M$. 
It is well known that if $\Char(k)$ does not divide $|G|$ (which is our
standing assumption) then
every $n$-dimensional $k[G]$-module $V$ lifts to a unique $A[G]$-module 
$V_A$, which is free of rank $n$ over $A$.  
 
It is shown in~\cite[Section 15.5]{serre77} that
the lifting operation $V \mapsto V_K := V_A \otimes K$
and the ``reduction mod $M$" operation
$V_K \mapsto V$ give rise to mutually inverse bijections between the 
representation rings $R_k(G)$ and $R_K(G)$ of $G$. These bijections 
send irreducible $k$-representations to irreducible $K$-representations 
of the same dimension, and they
are functorial in both $V$ and $G$. In particular, if $g \in G$
and $\zeta_d \in k$ is a primitive $d$th root of unity 
then the eigenspace $V(g, \zeta_d)$, viewed as a representation
of the cyclic subgroup $\langle g \rangle \subset G$, lifts
to $V_K(g, \eta_d)$ for some primitive $d$th root of unity $\eta_d \in A$ 
such that 
\begin{equation} \label{e.mixed}
\zeta_d = \eta_d \pmod{M}
\end{equation}
Taking $d = 1$, we see that if $g \in G$ acts on $V$ 
as a pseudo-reflection if and only if it acts on $V_K$ 
as a pseudo-reflection. 

This shows that for every pseudo-reflection group
$\phi \colon G \hookrightarrow \GL(V)$ over $k$
there is an abstractly isomorphic 
pseudo-reflection group $\phi_K \colon G \hookrightarrow \GL(V_K)$ over $K$.    
For each $g \in G$, the eigenvalues of $\phi(g)$ and $\phi_K(g)$ are 
the same, modulo $M$, in the sense that if $\eta_d$ is an eigenvalue of
$\phi_K(g)$ then $\zeta_d$ is an eigenvalue of $\phi(g)$, as in
\eqref{e.mixed}. Thus $\dim_k \, V(g, \zeta_d) = \dim_K \, V(g, \eta_d)$
and consequently,
\[ a_{\phi}(d) = \max_{g \in G} \, \dim_k \, V(g, \zeta_d)
= \max_{g \in G} \, \dim_K \, V_K(g, \eta_d)
= a_{\phi_K}(d) \]
for every $d \ge 1$.
Note also that the degrees of the fundamental invariants are the same
since they can be recovered from the $a(d)$'s as $d$ varies; 
cf.~\eqref{e.springer}.

We conclude that if $k$ is 
an algebraically closed field satisfying the above assumptions,
then many properties of irreducible pseudo-reflection 
groups, whose orders are prime to $\Char(k)$,
are the same over $k$ as they are over $\mathbb{C}$:
their isomorphism types, the numbers $a(d)$ for each
$d \ge 1$, the numbers of pseudo-reflections of each order, 
the number of central elements of each order, 
and the degrees of the fundamental invariants.
This allows us to use the Shephard-Todd classification
(e.g., from~\cite[Appendix D]{LTbook}, where $k$ is assumed to be $\bbC$) 
in our setting; cf. \cite[Section 15.3]{kane}.

We now proceed with the proof of Theorem~\ref{thm.horizontal1}(b).
Let $\phi \colon G \hookrightarrow \GL(V)$ be a minimal counterexample,
as in the statement of Proposition~\ref{prop.minimal}.
Then by Proposition~\ref{prop.minimal}, $\phi$ is irreducible.

\smallskip
{\bf The infinite families $\ST_1$ -- $\ST_3$.} 

Case $\ST_1$: Here $V$ is the natural $(n-1)$-dimensional representation of 
$G:=\Sym_n$. For $n \ge 3$, $G$ has trivial centre and hence, cannot 
be minimal by Proposition~\ref{prop.minimal}(c). For $n = 2$, $\dim(V) = 1$,
contradicting Proposition~\ref{prop.minimal}(a).

Case $\ST_2$: Here $G = G(m,l,n) \subset \GL_n$, where $m, n > 1$, $l$ divides $m$,
and $(m, l, n) \ne (2, 2, 2)$.  Here $G(m, l, n)$ is defined as 
a semidirect product of the diagonal subgroup 
\[ A(m, l, n) = \{ \diag(\zeta_m^{a_1} , \dots, \zeta_m^{a_n}) \, | \; \, 
a_1 + \dots + a_n \equiv 0 \! \! \! \pmod l \}  \subset \GL_n \]
and the symmetric group $\Sym_n$, whose elements
are viewed as permutation matrices in $\GL_n$;
see~\cite[Chapter 2]{LTbook}.
(Note that \cite{LTbook} assumes $k=\bbC$, but the same
construction works in our more general context.)
By Proposition~\ref{prop.minimal}(c), $G(m, l, n)$ contains
the scalar matrix $\zeta_p \id$. This matrix has to be contained in 
$A(m, l, n)$; hence, $p$ divides $m$.  Moreover 
by Proposition~\ref{prop.minimal}(i), we may assume $n \geqslant 2p$.
Consider
$g = \diag(\zeta_m^{m/p}, \ldots, \zeta_m^{m/p}, 1, \dots, 1)
\in A(m, l, n) \subset G(m, l, n)$, where $\zeta_m^{m/p}$ 
occurs $p$ times.  This element has order $p$ and exactly 
two eigenvalues, contradicting Proposition~\ref{prop.minimal}(d).  

Case $\ST_3$: Here $G$ is cyclic and $V$ is a $1$-dimensional. Once again,
this contradicts Proposition~\ref{prop.minimal}(a).

\smallskip
{\bf The exceptional cases $\ST_4$ -- $\ST_{37}$.}

All of the exceptional cases satisfy $\dim(V) \leqslant 8$.
On the other hand, by Proposition~\ref{prop.minimal}(h)~and~(i), 
$\dim(V) = mp$, where $m \geqslant 2$.  We conclude that 
either (I) $p = 2$ and $\dim(V) = 4$, $6$ or $8$, or
(II) $p = 3$ and $\dim(V) = 6$. 

Case I: We need to consider the groups $\ST_{28}$--$\ST_{32}$,
$\ST_{34}$, $\ST_{35}$, and $\ST_{37}$, with $p =2$.
With the exception of $\ST_{32}$, each of these groups
has a reflection of order $2$ and thus is ruled out 
by Proposition~\ref{prop.minimal}(e). The group 
$\ST_{32}$ is isomorphic to $\bbZ/3 \bbZ \times \Sp_4(\bbF_3)$
(see \cite[Theorem 8.43]{LTbook}).
The group $\Sp_4(\bbF_3)$ has non-central elements of order $2$, 
contradicting Proposition~\ref{prop.minimal}(f). 

Case II:
Here $p=3$ and we only need to consider two groups,
$\ST_{34}$ and $\ST_{35}$. 
The group $\ST_{35}$ has trivial centre and thus is ruled out by 
Proposition~\ref{prop.minimal}(c).  (Recall that the order 
of the centre is the greatest common divisor 
of the degrees $d_1, \dots, d_6$. For $\ST_{35} = W({\bf E_6})$
these are, $2$, $5$, $6$, $8$, $9$, and $12$.) This leaves us 
with $G = \ST_{34}$, otherwise 
known as the Mitchell group. The structure of this group was 
investigated by J.~H.~Conway and N.~J.~A.~Sloane. 
In~\cite[Section 2]{cs} they constructed four isomorphic lattices, 
$\Lambda^{(i)}$, 
where $i = 2$, $3$, $4$ and $7$, whose automorphism group is $\ST_{34}$.
In subsection 2.3 they showed that 
$\ST_{34} \simeq \Aut(\Lambda^{(3)})$ contains the group
$(2 \times 3^5) \rtimes \Sym_6$, which, in turn, contains
a $3$-group $H \simeq (3^2 \rtimes \langle (123) \rangle) \times 
(3^2 \rtimes \langle (456) \rangle) \simeq P \times P$, where $P$ 
is a non-abelian group of order $27$. 
By~\cite[Theorem 1.3]{mr10} (or, alternatively, 
by~\cite[Theorem 1.4(b)]{mr10}), $\ed(P) = 3$.
On the other hand, by~\cite[Theorem 4.1]{km}, $\ed(H; 3) = \ed(H)$, and
by~\cite[Theorem 5.1]{km}, 
$\ed(H) = \ed(P \times P) =  \ed(P) + \ed(P) = 6$.  
Since we are assuming that $\ST_{34}$,
with its natural $6$-dimensional representation, 
is a counterexample to Theorem~\ref{thm.horizontal1}(b), we
obtain
\[ 6= \ed(H) = \ed(H; 3) \leqslant \ed(\ST_{34}; 3) < a(3) = 6 \, . \]
This contradiction completes the proof of Theorem~\ref{thm.horizontal1}(b).
\qed

\section{A representation-theoretic corollary}

Before proceeding further we record a representation-theoretic corollary
of our proof of Theorem~\ref{thm.horizontal1}(b), which, to the best 
of our knowledge, has not been previously noticed. Recall that
$\rdim(H)$ denotes the minimal dimension of a faithful 
representation of a finite group $H$ over the base field $k$.

\begin{cor} \label{cor.main4}
Suppose $\zeta_p \in k$.  Let $G \subset \GL(V)$ 
be a finite subgroup generated by pseudo-reflections, 
$G_p$ be a $p$-Sylow subgroup of $G$, and $V' \subset V$ be 
a minimal (with respect to inclusion) faithful $k$-subrepresentation 
of $G_p$. Then $\dim(V') = \rdim(G_p)$.
\end{cor}
  
\begin{proof}
Since $\zeta_p \in k$, $\rdim(G_p) = \ed(G; p)$ by the Karpenko-Merkurjev 
theorem~\eqref{e.km}.  Choose $g$ as in Lemma~\ref{lem.main1}.  
Then, by Lemma~\ref{lem:pGroupMaximal},
\[ \ed(G; p) = \rdim(G_p) \leqslant \dim(V') \leqslant \dim \, V(g, \zeta_p)
= a(p) \, . \]
By Theorem~\ref{thm.horizontal1}(b), $\ed(G; p) = a(p)$ and 
thus the above inequalities are all equalities.
This completes the proof of Corollary~\ref{cor.main4}.
\end{proof}

The following example shows that Corollary~\ref{cor.main4} fails
if $G \subset \GL(V)$ is not assumed to be generated by pseudo-reflections.

\begin{example} \label{ex.main3} Let $p > 2$ be a prime,
$P$ be a non-abelian group of order $p^3$, and 
$\psi \colon P \hookrightarrow \GL(U)$ 
be a faithful $p$-dimensional representation of $P$.
Set $G = P \times P$ and 
\[ \phi = \psi_1 \otimes \psi_2 \oplus \psi_1 \colon 
G \to \GL(U \otimes U \oplus U) \, , \]
where for $i = 1, 2$, $\psi_i$ is the composition of $\psi$
with the projection $G \to P$ to the $i$th factor.
Both $\psi_1 \otimes \psi_2$ and $\psi_1$ are irreducible
representations of $G$; the irreducibility of
$\psi_1 \otimes \psi_2$ follows from~\cite[Theorem 3.2.10(i)]{serre77}.
These irreducible representations are distinct, because 
$\dim(\psi_1 \otimes \psi_2) = p^2$ and $\dim(\psi_1) = p$.

Note that $G = G_p$ is a group of order $p^6$, 
and $V=U \otimes U \oplus U$ is a faithful representation of $G$.
Since it is a direct sum of two distinct irreducibles, neither 
of which is faithful, the only faithful $G_p$-subrepresentation 
$V'$ of $V$ is $V$ itself.  On the other hand, $G$ has 
a $2p$-dimensional faithful representation $\psi_1 \oplus \psi_2$; 
hence, $\rdim(G) \leqslant 2p$.  In summary, $G = G_p$, $V = V'$ and
$\dim(V') = p^2+p > 2p \geqslant  \rdim(G_p)$.
Thus the assertion of Corollary~\ref{cor.main4} fails for 
$\phi(G) \subset \GL(V)$. 
\end{example}

\section{Proof of Theorem~\ref{thm.main3}(a)}

The degrees of the fundamental invariants 
of $W({\bf E_6})$ are $2, 5, 6, 8, 9$ and $12$; see, 
e.~g., \cite[p.~275]{LTbook}.
Thus by Theorem~\ref{thm.horizontal1}(b), $\ed(W({\bf E_6}); 2) = 4$.
This shows that $\ed( W({\bf E_6}) ) \geqslant 4$.

Recall that $\ed( W({\bf E_6}) )$ is the minimal value of $\dim(Y)$ 
such that there exists a dominant rational 
$W({\bf E_6})$-equivariant map $V \dasharrow Y$ defined over $k$,
where $V$ is a linear representation of $W({\bf E_6})$, and $Y$ is a 
a $k$-variety with a faithful action of $W({\bf E_6})$; see, 
e.g.,~\cite[Section 2]{icm}.  To prove the opposite 
inequality, $\ed( W({\bf E_6}) ) \leqslant 4$, it thus suffices 
to establish the following lemma suggested to us by I.~Dolgachev.

\begin{lem} \label{lem:WE6}
Let $k$ be a field of characteristic $\ne 2, 3$.
There exists a dominant $W({\bf E_6})$-equivariant map
\[ f \colon \bbA^6 \dasharrow Y \ , \]
defined over $k$,
where $\bbA^6$ is a linear representation of $W({\bf E_6})$ 
and $Y$ is a $4$-dimensional variety with
a faithful action of $W({\bf E_6})$.
\end{lem}

\begin{proof}
First, we construct $Y$. Consider the space $(\bbP^2)^6$ of ordered
$6$-tuples of points in the projective plane, and let
$U \subset (\bbP^2)^6$ be the dense open 
consisting of $6$-tuples $(a_1, \ldots, a_6)$ such that no two 
of the points $a_i$ lie on the same line, and no six lie on the same conic. 
This open subset is invariant under the natural (diagonal)  
$\PGL_3$-action on $(\bbP^2)^6$. Moreover, $U$ is contained 
in the stable locus of $(\bbP^2)^6$ for this action; 
see, e.g., \cite[p. 116]{do}. Thus there exists a geometric quotient
$q \colon U \to Y := U/\PGL_3$.  The explicit description
in~\cite[Example I.3]{do} show that
$Y$ and $q$ are defined over $k$.
Note that \[ \dim(Y) = \dim(U) - \dim(\PGL_3) = 
\dim \, (\bbP^2)^6 - \dim(\PGL_3) = 12-8 = 4,  \] 
as desired.

Now, we construct the affine space $\bbA^6$ and its map to $Y$.
Let $x, y, z$ be 
projective coordinates on $\bbP^2$ and $C \subset \bbP^2$  
be the cubic $y z^2 = x^3$. Note that $C$ has a cusp at $(0: 1: 0)$.
The smooth 
locus $C_{\rm sm} = C \setminus \{ (0:1:0) \}$ is an algebraic group  
isomorphic to the additive group $\bbG_a$. Indeed,
we identify $\bbG_a \simeq \mathbb A^1$ with $C_{\rm sm}$
 via $t \mapsto (t: t^3: 1)$.
Thus the space $C_{\rm sm}^6$ is isomorphic to affine space $\bbA^6$.

This yields a rational map
\[ \phi \colon C_{\rm sm}^6 \to C^6 \hookrightarrow (\bbP^2)^6 \, . \] 
Three points $t_1, t_2 , t_3 \in C_{\rm sm}$ lie on a line if and only
if $t_1 + t_2 + t_3 = 0$; six points $t_1, \dots, t_6
\in C_{\rm sm}$ lie on a conic if and only $t_1 + \dots + t_6 = 0$.
Thus for general $(t_1, \dots, t_6) \in C_{\rm sm}^6$,
we have $\phi(t_1, \dots, t_6) \in U$.
In other words, we may view $\phi$ as a rational map 
$C_{\rm sm}^6 \dasharrow U$.
We now define the map $f \colon C_{\rm sm}^6 \dasharrow Y$ as the composition
\[ f \colon C_{\rm sm}^6 \stackrel{\phi}{\dasharrow} U \stackrel{q}{\to} Y \, . \] 
By \cite[Lemma 13]{shioda}, over the algebraic closure,
if $(t_1, \dots, t_6)$ is a $6$-tuple 
of points in general position in $\mathbb{P}^2$ then there is 
a cuspidal cubic $C' \subset \bbP^2$ such that $t_1, \dots, t_6$ 
lie in the smooth locus of $C'$. Since any two cuspidal cubics in $\bbP^2$
are projectively equivalent
(recall our assumptions on the characteristic),
we conclude that $f$ is dominant.

It remains to construct actions of $W({\bf E_6})$ on $\bbA^6$ and $Y$, 
and to show that $f$ is equivariant.
Recall that blowing up 6 points in $\bbP^2$ produces a cubic surface
$X$ with the 6 exceptional divisors of the blow-up corresponding to
a ``sixer'': 6 pairwise disjoint lines in $X$.
Conversely, any sixer can be blown down to produce 6 points on $\bbP^2$.
Over an algebraically closed field,
the elements of $W({\bf E_6})$ act freely and transitively on the set of
sixers in $X$ (where we keep track of the ordering of the 6 lines).
This produces a faithful action of $W({\bf E_6})$ on $Y$
which is defined over $k$.
This action of the Weyl group $W({\bf E_6})$ on $Y$ is sometimes
called the {\em Cremona representation} or the {\em Coble representation}.
For more details, see~\cite[Section 7]{dol83}, \cite[Section 6]{dol08},
and~\cite[Chapter 6]{do}.

We recall how $W({\bf E_6})$ acts on the Picard group $N$ of a
smooth cubic surface $X \subset \bbP^3$
over an algebraically closed field;
see, e.g., \cite[Sections 4 and 5]{dol83}
or~\cite[Section 26]{manin}. 
The Picard group $N \simeq \bbZ^7$ with its intersection form
is a lattice with a symmetric
bilinear form given by $\diag(1,-1,\ldots,-1)$ with respect to the
basis $e_0, \ldots, e_6$,
where $e_0$ is the hyperplane section of $X$ and
$e_1, \ldots, e_6$ is a collection of 6 mutually disjoint lines on $X$.

Consider a set of fundamental roots in $N$ given by
\[ \alpha_1 = e_0-e_1-e_2-e_3,
\quad \alpha_{2}=e_{2}-e_1,
\quad \ldots
\quad \alpha_{6}=e_{6}-e_5 \ . \]
The reflections associated to these roots generate a group isomorphic
to $W({\bf E_6})$. (Note that $\alpha_1, \ldots, \alpha_6$ are 
the same as the fundamental roots used by I.~Dolgachev in~\cite{dol83}, 
up to reordering, and as the fundamental roots used by Yu.~Manin 
in~\cite{manin}, up to sign; see~\cite[Proof of Proposition 25.2]{manin}.)
The reflections 
associated to $\alpha_2, \ldots, \alpha_6$
generate a subgroup isomorphic to $\Sym_6$ which permutes
the basis elements $e_1, \ldots, e_6$.
The symmetric group $\Sym_6$ naturally acts on $C_{\rm sm}^6$
and $(\bbP^2)^6$ by permutations; thus $f$ is $\Sym_6$-equivariant.
It remains to consider the reflection $g \in W({\bf E_6})$ associated to
the root $\alpha_1$.

First, we identify the action of $g$ on $Y$.  Suppose
$\pi : X \to \bbP^2$ is the blowup of 6 points $a_1, \ldots, a_6$.
Identifying each $e_i$ with the class of each exceptional divisor
$E_i := \pi^{-1}(a_i)$ in the cubic surface $X$ we may determine the
action of $g$.  Indeed, for $i \ne j \ne k$ taken from $\{1,2,3\}$,
the line $E_i$ is taken to the strict transform of
the line between $a_j$ and $a_k$; while $E_4$, $E_5$, $E_6$ are all left
fixed.
Recall that the standard quadratic transform $s : \bbP^2 \dasharrow
\bbP^2$ at the points $a_1, a_2, a_3$ is the map obtained by blowing up
the points and then blowing down the strict transforms of the lines
between them.  In this language, $g : Y \to Y$ is given by
\[ [a_1, \ldots, a_6] \mapsto
[s(a'_1),s(a'_2),s(a'_3),s(a_4),s(a_5),s(a_6)] \]
where
$a'_1$ is any point on the line between $a_2$ and $a_3$ (and similarly
for $a'_2$ and $a'_3$).

We now construct an action of $g$ on $C_{\rm sm}^6$ 
following H.~Pinkham~\cite{pinkham}.
If $C \subset \bbP^2$ is a cuspidal cubic, then, for any three points
$u_1$, $u_2$ and $u_3$ in the smooth locus $C_{\rm sm}$ of $C$,
$C'=s(C)$ is also a cuspidal cubic in $\bbP^2$. Since any two cuspidal cubics
in $\bbP^2$ are linear translates of each other, there exists an $l \in \PGL_3$
such that $l(C') = C$.  Composing $s$ with $l$, one obtains 
a rational map $l \cdot s : C_{\rm sm} \dasharrow C_{\rm sm}$ which 
is regular on $C_{\rm sm} \setminus \{ u_1, u_2, u_3 \}$.
Let $u_1'$ be the unique third intersection point of $C$ with the line
passing through $u_2$ and $u_3$ (and similarly for $u'_2$ and $u'_3$).
We define a map $g : C_{\rm sm}^6 \to C_{\rm sm}^6$ via
\[ (u_1, \ldots, u_6) \to
(l \cdot s (u'_1), l \cdot s(u'_2), l \cdot s(u'_3), l \cdot s(u_4), 
l \cdot s(u_5), l \cdot s(u_6)) \ . \]
By construction, we see that $f$ is $g$-equivariant.

Note that the choice of $l$ and thus of the map 
$l \cdot s : C_{\rm sm} \dasharrow C_{\rm sm}$ above is not unique.
Pinkham's observation~\cite[pp. 196--197]{pinkham}
is that there is a choice of $l$ such that the resulting map
$g$ gives rise to a linear representation of 
$W({\bf E_6}) = \langle g, \Sym_6 \rangle$ on $C_{\rm sm}^6 \simeq \bbA^6$.
In fact, $C_{\rm sm}^6$ can be identified with a Cartan subalgebra
of the Lie algebra of type ${\bf E_6}$ with the standard action of the
Weyl group.
This construction is valid over any
field $k$ of characteristic $\ne 2, 3$.
This completes the proof of Lemma~\ref{lem:WE6} and thus
of Theorem~\ref{thm.main3}(a).
\end{proof}

\section{Proof of Theorem~\ref{thm.main3}(b) and (c)}

As we have previously pointed out,
$\ed(G) \leqslant \dim(V)$; see, e.g.,~\cite[(2.3)]{icm}.
In the case where $G = G(m, m, n)$ and $m \geqslant 2$ and $(m, n)$ are 
relatively prime, no element of $G$ acts as a scalar on $V$.
The natural $G$-equivariant dominant rational map $V \dasharrow \bbP(V)$
tells us that $\ed(G) \leqslant \dim(V) - 1$.

It now suffices to show that for every irreducible $G \subset \GL(V)$
generated by pseudo-reflections there exists a prime $p$
such that
\[ a(p) =  \begin{cases} \dim(V)-1,& \text{if $G \simeq G(m,m,n)$ with
$m$, $n$ relatively prime,} \\
\dim(V),& \text{otherwise.}
\end{cases} \]
Indeed, Theorem~\ref{thm.horizontal1}(b) will then tell us that
$\ed(G) \geqslant \ed(G; p) \geqslant a(p) \geqslant \dim(V) - 1$ in the first case and
$\ed(G) \geqslant \ed(G;p) \geqslant a(p) \geqslant \dim(V)$ in the second. 
Since we 
have established the opposite inequalities, this will complete 
the proof in both cases.

By Springer's theorem~\eqref{e.springer},
$a(p)$ is equal to the number of invariant degrees $d_i$
which are divisible by $p$.  In the case
where $G = G(m, m, n)$, $m \geqslant 2$ and $(m, n)$ are 
relatively prime, the degrees $d_i$ are $m, 2m, \dots, (n-1)m$,
and $n$. Taking $p$ to be a prime divisor of $m$, we see that 
$a(p) = n-1 = \dim(V) - 1$, as desired.

For all other groups of the form $G = G(m, l, n)$, with $m \geqslant 2$ the degrees 
$d_i$ are $m, 2m, \dots, (n-1)m$, and $\dfrac{mn}{l}$. All of them
are divisible by every prime factor $p$ of $\gcd(m,\frac{mn}{l}) > 1$. 
Hence, in this case $a(p) = n = \dim(V)$, as desired.

Finally, in the case where $m = 1$, $G(m, l, n) = G(1, 1, n) = \Sym_n$ is
excluded by our hypothesis.

This leaves us with the exceptional groups $ST_4$ -- $ST_{37}$.
If $G \ne \ST_{25}, \ST_{35}$ then every 
degree $d_i$ of $G$ is divisible by $2$.  If $G = \ST_{25}$ then every 
degree $d_i$ of $G$ is divisible by $3$.
Finally, $\ST_{35} = W({\bf E_6})$ was treated in part (a).
\qed

\medskip
\begin{remark} Our proof shows that for every $G$ 
in the statement of Theorem~\ref{thm.main3}
there is a prime $p$ such that $\ed(G) = a(p) = \ed(G; p)$.
\end{remark}

\begin{remark}
Pinkham's construction applies in greater generality than
the case of $W({\bf E_6})$ used in Lemma~\ref{lem:WE6}.  
In particular, one can use it to construct a dominant rational
$W({\bf E_7})$-equivariant map $\bbA^7 \dasharrow Z$, where $Z$ is a
dense open subset of the $6$-dimensional variety $(\bbP^2)^7_{ss}//\PGL_3$. 
Here the subscript $_{ss}$ denotes the semistable locus.
Since we know that $\ed(W({\bf E_7})) = 7$ by Theorem~\ref{thm.main3}(c),
this gives an alternative (indirect) proof of the classical fact that
the Coble representation of $W({\bf E_7})$ on 
$(\bbP^2)^7_{ss}//\PGL_3$ is not faithful; see~\cite[p. 293]{dol83} 
or~\cite[p. 122]{do}.
\end{remark}

\section{A variant of Bertini's theorem} 
\label{sect.bertini}

Our proof of Theorem~\ref{thm.horizontal2} will rely on
the following variant of Bertini's theorem.

\begin{thm} \label{thm.bertini}
Let $Y$ be a smooth, geometrically irreducible subscheme of the space
$\bbP^N := \operatorname{Proj} (k[y_0, \dots, y_N])$, 
$C \subset Y$ be a smooth $0$-dimensional closed subscheme of $Y$,
$X$ be a geometrically irreducible variety, and
$\psi \colon X \to Y$ be a smooth morphism,
all defined over $k$. Assume that $\dim(Y) \geqslant 2$.
When $k$ is an infinite field of positive
characteristic, we also assume that $\psi$ is \'etale.

Given a homogeneous polynomial $f \in k[y_0, \ldots, y_N]$, 
let $\Yfnote{Y}{f}$ be
the intersection of $Y$ with the
hypersurface $\{ f=0 \}$
and let $\Yfnote{X}{f}$ denote the preimage of $\Yfnote{Y}{f}$ under $\psi$.
Then for $a \gg 0$ there exists a homogeneous polynomial $f$ of 
degree $a$ satisfying the following conditions:

\smallskip
(i) $\Yfnote{X}{f}$ is geometrically irreducible,

\smallskip
(ii) $\Yfnote{Y}{f}$ is smooth,

\smallskip
(iii) $\Yfnote{Y}{f}$ contains $C$,

\smallskip
(iv) $\dim(\Yfnote{X}{f})=\dim(X)-1$.
\end{thm}

In the case where $k$ is infinite, Theorem~\ref{thm.bertini} 
can be deduced from the classical Bertini theorem. In the situation 
where $X = Y$ and $\psi = \operatorname{id}$, this is done 
in~\cite{kleiman-altman}. A similar argument can be used to
prove Theorem~\ref{thm.bertini} in full generality (here $k$ is 
still assumed to be infinite). For the sake of completeness
we briefly outline this argument below.

\begin{proof}[Proof of Theorem~\ref{thm.bertini} in the case 
where $k$ is an infinite field] 
Denote the ideal of $C \subset \bbP^N$
by $\calI \subset k[y_0, \dots, y_N]$. Let
$\calI_a$ be the homogeneous part of $\calI$ of degree $a$.
For $f \in \calI_a$ in general position,
$\Yfnote{Y}{f}$ is smooth at $C$ and of dimension $\dim(Y) - 1$. Now consider the map
\[ \phi_a \colon X \setminus \psi^{-1}(C) \to \bbP(\calI_a) \]
obtained by composing $\psi$ with the morphism 
$\iota \colon Y \setminus C \to \bbP(\calI_a)$, given by the linear 
system of degree $a$ hypersurfaces passing through $C$. 
(Note that $\iota$ is an embedding for $a \gg 0$.) By Bertini's 
Smoothness Theorem~\cite[Corollaire 6.11(2)]{Jou83Theoremes}, 
for $f \in \calI_a$ in general position, $\Yfnote{Y}{f}$ is smooth
away from $C$. Since $\Yfnote{Y}{f}$ is also smooth at $C$,
we conclude that $\Yfnote{Y}{f}$ is smooth,
every irreducible component of $\Yfnote{Y}{f}$ is of dimension 
$\dim(Y) - 1$ and hence, every irreducible component of
$\Yfnote{X}{f}$ is smooth of dimension $\dim(X) - 1$.
By Bertini's Irreducibility Theorem~\cite[Corollaire 6.11(3)]{Jou83Theoremes}, 
for $f \in \calI_a$ in general position, $\Yfnote{X}{f} \setminus \psi^{-1}(C)$ 
is geometrically irreducible of dimension
$\dim(X) - 1$.  (This is where the assumption that $\psi$ is
\'etale is used when $k$ is of positive characteristic.)
Since $\dim(Y) \geqslant 2$, we have
$\dim(X)-\dim(Y) \leqslant \dim(X) - 2$ and
thus $\psi^{-1}(C)$ cannot contain a component of $\Yfnote{X}{f}$. 
Hence, $\Yfnote{X}{f}$ itself is geometrically irreducible.
This completes the proof of Theorem~\ref{thm.bertini} in the 
case where $k$ is infinite.
\end{proof}

If $k$ is a finite field, the classical Bertini theorems break down.
In this case our proof will be based on the
probabilistic versions of Bertini's smoothness and
irreducibility theorems, due to B.~Poonen~\cite{poonen2} and 
F.~Charles and B.~Poonen~\cite{CPbertini}, respectively.

We begin by recalling the notion of density from~\cite{poonen1}.  
Let $\calS = k[y_0, \dots, y_N]$ be the homogeneous coordinate 
ring of $\bbP^N$, $\calS_a \subset \calS$ be the $k$-vector 
subspace of homogeneous polynomials of degree $a$, and
$\calS_{\rm hom} = \cup_{a \ge 0} \,  \calS_a$.
The {\em density} $\mu(\calP)$ of any subset 
$\calP \subset \calS_{\rm hom}$ is defined as
\[ \mu(\calP) := \lim_{a \to \infty} \frac{|\calP \cap 
\calS_a|}{|\calS_a|} \, . \]
Note $\mu(\calP)$ is either a real number between $0$ and $1$
or undefined (if the above limit does not exist).

\begin{lem} \label{lem.density} Suppose 
$\calP_1, \calP_2 \subset \calS_{\rm hom}$. If
$\mu(\calP_1) = 1$ then
$\mu(\calP_1 \cap \calP_2) = \mu(\calP_2)$.
\end{lem}

\begin{proof} The lemma is a consequence of the inequalities
\[ |\calP_2 \cap \calS_a| - |\calS_a \setminus \calP_1| \leqslant 
|\calP_1 \cap \calP_2 \cap \calS_a| \leqslant
|\calP_2 \cap \calS_a| \, , \]
since 
$\lim_{a \to \infty} \dfrac{|\calS_a \setminus \calP_1|}{|\calS_a|} = 0$. 
\end{proof}

\begin{proof}[Proof of Theorem~\ref{thm.bertini} in the case 
where $k$ is a finite field]
Let $\calS := k[y_0, \dots, y_N]$ and $\calI$ be the ideal in
$\calS$ corresponding to $C \subset \bbP^N$;
and let $\calS_{\rm hom}$, 
$\calI_{\rm hom}$ be the set of homogeneous polynomials in 
$\calS$, $\calI$, respectively. 

We define $\calP_1$ as the set 
of $f \in \calS_{\rm hom}$ such 
that $\Yfnote{X}{f}$ is geometrically irreducible, and $\calP_2$ as the set 
of $f \in \calI_{\rm hom}$ 
such that $\Yfnote{Y}{f}$ is smooth and $\dim(\Yfnote{Y}{f}) = \dim(Y)-1$.
Thus $\calP_1 \cap \calP_2$
is precisely the set of homogeneous polynomials satisfying 
conditions (i), (ii), (iii) and (iv) of the theorem. 
Our goal is to show that $\mu(\calP_1 \cap \calP_2)$ exists and
is $ > 0$. If we can prove this, the theorem will immediately follow.  

Since we are assuming that all fibers of $\psi$ have dimension 
$\le \dim(X) - 2$, no irreducible component of $\Yfnote{X}{f}$ 
can be contained in
a fiber of $\psi$.  Thus, by \cite[Theorem 1.6]{CPbertini}, 
\[ \mu(\calP_1) = 1 \, . \]
Note that here we use the assumption that $\dim(Y) \ge 2$. 
On the other hand, by \cite[Theorem 1.1]{poonen2}, the local density 
\[ \mu_C(\mathcal{P}_2) = 
\lim_{a \to \infty} \frac{|\calP_2 \cap \calI_a|}{|\calI_a|} 
\; \; \text{exists and is $> 0$.} \]
(This uses our assumptions that $C$ is smooth and $0$-dimensional. 
In particular,  $\dim(X) > 2 \dim(C)$.)
Since $C$ is a zero-dimensional subscheme of $\bbP^n$,
we have $\dim_k(\calI_a) = \dim_k(\calS_a) - \deg(C)$,
for large $a$. Here $\deg(C)$ denotes the degree of $C$ in $\bbP^n$. Thus
\[ \lim_{a \to \infty} \frac{|\calI_a|}{|\calS_a|} = |k|^{-\deg(C)} > 0 \, . \]
Since $\calP_2$ is, by definition, a subset of $\calI_{\rm hom}$, we
have $\calP_2 \cap \calI_a = \calP_2 \cap \calS_a$ and thus 
\[ \mu(\mathcal{P}_2) = 
\lim_{a \to \infty} \frac{|\calP_2 \cap \calS_a|}{|\calS_a|} = 
\lim_{a \to \infty} \frac{|\calP_2 \cap \calI_a|}{|\calI_a|} \cdot 
\frac{|\calI_a|}{|\calS_a|}  \; \;
\text{also exists and is $> 0$.} \]  
Lemma~\ref{lem.density} now tells us that
$\mu(\calP_1 \cap \calP_2)$ exists and is $> 0$, as desired.
\end{proof}

\section{Proof of Theorem~\ref{thm.horizontal2}: preliminaries}

First we observe that part (b) is an immediate consequence of part (a).
Indeed, combining the first inequality in~\eqref{e.obvious} with
part (a), we have
\[ \max_p \, \ed(G; p) \leqslant \pmed(G) \leqslant \max_p \, a(p) \, , \]
Theorem~\ref{thm.horizontal1}(b) now tells us that $a(p) = \ed(G; p)$
for each prime $p$, and part (b) follows.

 From now on we will focus on the proof of Theorem~\ref{thm.horizontal2}(a).
Let $G$ be a finite group and
$G \hookrightarrow \GL(V)$ be a faithful linear representation 
defined over $k$.
We will assume throughout that $\Char(k)$ does not divide $|G|$.
Consider the closed subscheme
\[ B := \bigcup_{\scriptstyle  g \in G, \; \zeta \ne 1 } \, 
V(g, \zeta)
\quad
\text{or equivalently,}
\quad  
B = \bigcup_{
\scriptstyle
\begin{array}{c} \scriptstyle g \in G, \; \zeta^p = 1 \\ 
 \scriptstyle \zeta \ne 1, \;   
\text{ $p$ prime}
\end{array}} 
\,  V(g, \zeta) \, , \]
where $\zeta$ ranges over the roots of unity in $\bar{k}$.
Note that, although each $V(g, \zeta)$ is defined only over
$k(\zeta)$, their union $B$ is defined over $k$.

The following lemma may be viewed as a variant 
of~\cite[Proposition 3.2]{Spr74Regular}. 

\begin{lem} \label{lem.invariant1}
Let $m \ge |G|$ be an integer.  Suppose $v \in V$ 
has the property that $f(v) = 0$ for every $G$-invariant
homogeneous polynomial $f \in k[V]$ of degree $m$.  Then $v \in B$.
\end{lem}

\begin{proof} We may assume $v \ne 0$. Let 
$\overline{v} \in \bbP(V)$ be the projective point associated to $v$. 
Denote the $G$-orbit of $\overline{v}$ by $\overline{v}_1 = \overline{v}$, 
$\overline{v}_2, \dots, \overline{v}_r \in \bbP(V)$. Note that 
$r \le |G| \le m$. 

We claim that there exists a homogeneous polynomial $h \in k[V]$ 
of degree $m$ such that $h(\overline{v}_1) \ne 0$ but 
$h(\overline{v}_i) = 0$ for any $i = 2, \dots, r$.
To construct $h$, for every $i = 2, \dots, r$ choose a linear form 
$l_i \in V^*$ such that $l_i(\overline{v}_i) = 0$ but 
$l_i(\overline{v}_1) \ne 0$. 
Now set $h = l_2^{m + 2 - r} l_3 \dots l_r$. This proves the claim.

We now define a $G$-invariant homogeneous polynomial $f$ of degree $m$ 
by summing the translates of $h$ over $G$:
\begin{equation} \label{e.average}
f(v') = \sum_{g \in G} h(g \cdot v') \quad \forall v' \in V \, .
\end{equation}
By our assumption, $f(v) = 0$.

Let $S \subset G$ be the stabilizer of $\overline{v}$, i.e., the subgroup of
elements $s \in G$ such that $v$ is an eigenvector for $s$. Then
$s(v) = \chi(s) v$ for some multiplicative character $\chi \colon S \to k^*$.
It now suffices to show that $\chi(s) \ne 1$ for some $s \in S$.
Indeed, if we denote $\chi(s)$ by $\zeta$, for this $s$, then
$v \in V(s, \zeta) \subset B$, as desired. 

To show that $\chi(s) \ne 1$ for some $s \in S$, recall that
by our choice of $h$, $h(g \cdot v) = 0$ unless $g \in S$.
Thus 
\[ 0 = f(v) = \sum_{s \in S} \, h(s \cdot v) = 
\sum_{s \in S} \, h(\chi(s) v) =  
 \sum_{s \in S} \, \chi(s)^m  h(v) \, . \]
If $\chi(s) = 1$ for every $s \in S$, this yields
$0 = |S| \cdot h(v)$. This is a contradiction since $h(v) \ne 0$, and
we are assuming that $\Char(k)$ does not divide $|G|$.
Thus $\chi(s) \ne 1$ for some $s \in S$, as claimed.
\end{proof}

Denote the direct sum of $V$ and the trivial $1$-dimensional 
representation of $G$ by $W := V \times k$. 
Let $z$ be the coordinate along the second
factor in $W = V \times k$. We will identify 
$V$ with the open subvariety of $\bbP(W)$ given by $z \ne 0$, and
$\bbP(V)$ with the closed subvariety of $\bbP(W)$ given by $z = 0$.
Set $n := \dim(V)= \dim(\bbP(W))$.
If $C$ is a cone in $V$ with vertex at the origin,
we will denote by $\bbP(C)$ the image of 
$C \setminus \{ 0 \}$ under the natural projection
$(V \setminus \{ 0 \}) \to \bbP(V)$.

\begin{prop} \label{prop.invariant3} 
Consider the rational map
\[ \psi_m \colon \bbP(W) \dasharrow \bbP^N \]
given by the linear system $k[W]^G_m$ of $G$-invariant 
homogeneous polynomials of degree $m$ on $W$.
Denote the closure of the image of $\psi_m$ by 
$Y \subset \bbP^N$.
Assume $m \geqslant |G|$.  Then:

\smallskip
(a) The map $\psi_m$ is regular away from $\bbP(B)$.

\smallskip
(b) $\psi_m \colon \bbP(W) \dasharrow Y$ 
induces an isomorphism between $k(Y)$ and
the field of $G$-invariant rational functions on $\bbP(W)$.

\smallskip
(c) For a prime $q \gg 0$,
every fiber of the morphism $\psi_q \colon \bbP(W \setminus B) \to Y$ 
is finite.
\end{prop}

\begin{proof} 
(a) We may assume without loss of generality that
$k$ is algebraically closed.
Since $z^m \in k[W]^G_m$, we see that the indeterminacy 
locus of $\psi_m$ consists of points $(v: a) \in \bbP(W)$
with $a = 0$ and $f(v) = 0$ for every $f \in k[V]^G_m$, where
$k[V]^G_m$ denotes the $k$-vector space of $G$-invariant homogeneous 
polynomials on $V$ of degree $m$. By Lemma~\ref{lem.invariant1},
$v \in B$. Thus 
$(v:a) \in \bbP(B) \subset \bbP(V \times \{ 0 \}) \subset \bbP(W)$, as
claimed.

(b) To show that the natural inclusion
$\psi_m^* \colon k(Y) \hookrightarrow k(\bbP(W))^G$ of fields
is an isomorphism, we restrict  
$\psi_m$ to the dense open subset $V \subset \bbP(W)$
given by $z \neq 0$. This restriction is the morphism 
\[ \begin{array}{c} V \to \bbA^{N} \\
v \mapsto (f_1(v), \dots, f_{N}(v)) \, , 
\end{array} \] 
where $f_1, \dots, f_{N}$ form a basis of the vector space 
$k[V]^G_{\leqslant m}$ of $G$-invariant polynomials of degree $\leqslant m$.
Consequently, $f_1, \dots, f_{N} \in \psi_m^* \, k(Y)$.
By the Noether bound $k[V]^G$ is generated by polynomials of degree
$\leqslant |G|$ as a $k$-algebra; see Remark~\ref{rem.noether} below.
Since $|G| \leqslant m$, we conclude that
$\psi_m^* \, k(Y)$ contains $k[V]^G$ and thus its fraction field
$k(V)^G$. Since $V$ is a $G$-invariant dense open subset of $\bbP(W)$, we
have $k(V) = k(\bbP(W))$. Therefore, 
$\psi_m^* \, k(Y) \supset k(V)^G = k(\bbP(W))^G$, as desired.

\smallskip
(c) Suppose $v \in V \subset \bbP(W)$, i.e., $z(v) \ne 0$.
The argument of part (b) shows that in this case $w$ lies 
in the same fiber of $\psi$ as $v$ if and only if $w \in V$ and
$f(v) = f(w)$ for every $f \in k[V]^G$. Since elements of $k[V]^G$
separate the $G$-orbits in $V$, this shows that the fibers of $\psi_q$ 
in $V$ are precisely the $G$-orbits in $V$, and hence, are finite.

We may thus restrict $\psi_q$ to $\bbP(V) \subset \bbP(W)$, where $z = 0$.
That is, it suffices to show that if $q$ is a large enough prime,
every fiber of the morphism
$\psi_q \colon \bbP(V \setminus B) \to \bbP^N$ is finite. 
Equivalently, it suffices to show that every fiber of 
the morphism 
\[ \phi_q \colon V \setminus B \to \mathbb A(k[V]^G_{q}) \] 
given by the linear system $k[V]^G_{q}$ 
of $G$-invariant polynomials of degree $q$, is finite.
In particular, we may assume without loss of generality that
$B \subsetneq V$.

Choose homogeneous generators $g_1, \dots, g_r$ for
$k[V]^G$ and fix them for the rest of the proof. 
Denote their degrees by $d_1, \dots, d_r$, respectively.
By the Noether bound we may assume that $d_1, \dots, d_r \leqslant |G|$.

Let $\Lambda_{d_1, \dots, d_r}^q \subset \bbZ_{\ge 0}^r$ 
be the set of non-negative integers solutions $(a_1, \dots, a_r)$ 
of the linear Diophantine equation
\[ a_1 d_1 + \dots + a_r d_r = q.  \]
Then the polynomials $g_1^{a_1} \dots g_r^{a_r}$ 
span $k[V]_{q}^G$, as $(a_1, \dots, a_r)$ ranges 
over $\Lambda_{d_1, \dots, d_r}^q$.
In other words $\phi_q(v) = \phi_q(w)$ if and only if
$g_1^{a_1}(v) \dots g_r^{a_r}(v) = g_1^{a_1}(w) \dots g_r^{a_r}(w)$ 
for every $(a_1,\dots, a_r) \in \Lambda_{d_1, \dots, d_r}^q$.

Let us now fix $v \in V \setminus B$ and consider 
$w \in V \setminus B$ such that $\phi_q(w) = \phi_q(v)$.
Our ultimate goal is to show that, if $q$ is a large enough prime,
there are only finitely many such $w$.
After renumbering $g_1, \dots, g_r$, we may assume that $g_1(v), \dots, g_s(v)
\neq 0$ but $g_{s+1}(v) = \dots = g_r(v) = 0$.

\smallskip
Claim 1: $d_1, \dots, d_s$ are relatively prime. 

\smallskip
Indeed, assume the contrary: $\gcd(d_1, \dots, d_s) \geqslant 2$.
Choose a prime $q > |G|$.
Since $v \not \in B$, Lemma~\ref{lem.invariant1} tells us that
there exists an $f \in k[V]_q^G$ such that $f(v) \neq 0$.
Since $f$ is a polynomial in $g_1, \dots, g_r$, some monomial
$g_1^{a_1} \cdot \dots \cdot g_r^{a_r}$ of total degree 
$a_1 d_1 + \dots + a_r d_r = q$ does not vanish at $v$. 
After replacing $f$ by this monomial, we may assume that
$f = g_1^{a_1} \cdot \dots \cdot g_r^{a_r}$.
Note that $a_j \ge 1$ for some $j \ge s+1, \dots, r$. Otherwise
$q$ would be divisible by $\gcd(d_1, \dots, d_s)$, which 
is not possible, because $q$ is a prime and
$q > |G| \geqslant d_1, \dots, d_s \geqslant
\gcd(d_1, \dots, d_s) \geqslant 2$.
Since $g_j(v) = 0$, we conclude that
$f(v) = g_1^{a_1}(v) \cdot \dots \cdot g_r^{a_r}(v)  = 0$, a contradiction.
This completes the proof of Claim 1.

\smallskip
It is well known that if $d_1, \dots, d_s \geqslant 1$ are relatively 
prime integers then for large enough integers $q$ (not necessarily prime),
 $\Lambda_{d_1, \dots, d_s}^q \ne \emptyset$. 
The largest integer $q \geqslant 0$ such 
that $\Lambda_{d_1, \dots, d_s}^q = \emptyset$ is
called {\em the Frobenius number}; we will denote it by $F(d_1, \dots, d_s)$. 
This number has been extensively studied; for an explicit upper 
bound on $F$ in terms of $d_1, \dots, d_s$, see, e.g.,~\cite{erdos-graham}.

\smallskip
Claim 2: Suppose our prime $q$ is 
$> |G| + F(d_1, \dots, d_s) + d_1 + \dots + d_s$. Then

\noindent
(i) $g_i(w) \ne 0$ for every $i = 1, \dots, s$ and
(ii) $g_j(w) = 0$ for every $j = s+1, \dots, r$.

\smallskip
To prove (i), note that since $q - d_1 - \dots - d_s > F(d_1, \dots, d_s)$,
there is an $s$-tuple 
$(a_1, \dots, a_s)$ of non-negative integers such that
$a_1 d_1 + \dots + a_s d_s = q - d_1 - \dots - d_s$. Thus
the polynomial $P := g_1^{a_1 + 1} \cdot \ldots \cdot g_s^{a_s + 1}$ lies in
$k[V]_q^G$. By our assumption, $P(w) = P(v) \neq 0$. Hence,
$g_i(w) \neq 0$ for any $i = 1, \dots, s$.

To prove (ii), choose $j$ between $s+1$ and $r$. Since
$q > |G| + F(d_1, \dots, d_s) \geqslant d_j + F(d_1, \dots, d_s)$, 
there is an $s$-tuple 
$(b_1, \dots, b_s)$ of non-negative integers such that
$b_1 d_1 + \dots + b_s d_s = q - d_j$. Now
the polynomial $Q := g_1^{b_1} \cdot \ldots \cdot g_s^{b_s} g_j$ lies in
$k[V]_q^G$. Since $g_j(v) = 0$, we have $Q(w) = Q(v) = 0$.
By (i), $Q(w) = 0$ is only possible if $g_j(w) = 0$.
This completes the proof of Claim 2.

\smallskip
Claim 3.  There exists an $q_0 > 0$ such that for any integer
$q \ge q_0$ (not necessarily a prime), 
the set $\Lambda^q_{d_1, \dots, d_s}$ spans $\bbQ^s$ as a $\bbQ$-vector space.

\smallskip
To prove Claim 3, choose an integer basis 
$\vec{z}_1, \dots \vec{z}_{s-1} \in \bbZ^s$ for the $\bbQ$-vector space of
solutions of the homogeneous linear equation 
$a_1 d_1 + \dots + a_s d_s = 0$.  Denote the maximal absolute value 
of the coordinates of $\vec{z}_1, \dots \vec{z}_{s-1}$ 
by $M$ and set $q_0 := F(d_1, \dots, d_s) + (d_1 + \dots + d_s) M$.

For every $q > q_0$ we will construct 
an $\vec{a} = (a_1, \dots, a_s) \in \Lambda^q_{d_1, \dots, d_s}$ such 
that $a_i \geqslant M$ for every $i$.
Indeed, since $q - (d_1 + \dots + d_s) M > F$ there are non-negative
$b_1, \dots, b_s$ such that 
$b_1 d_1 +  \dots +  b_s d_s = q - (d_1 + \dots + d_s) M$.
We can now take $\vec{a} := (b_1 + M, \dots, b_s + M)$.

Finally,  for $q > q_0$, the $s$ integer vectors
\[ \vec{a}, \vec{a} + \vec{z}_1, \dots, \vec{a} + \vec{z}_{s-1} \]
lie in $\Lambda_{d_1, \dots, d_s}^q$ and are linearly independent.
This completes the proof of Claim 3.

\smallskip
Suppose $q$ is a prime, large enough to satisfy 
the assumptions of Claims 2 and 3.
We are now in a position to show that for any $v \in V \setminus B$,
there are only finitely many $w \in V \setminus B$ such that
$\phi_q(v) = \phi_q(w)$. By Claim 3, there exist $s$ linearly independent 
vectors
$(a_{11}, \dots, a_{1s}), \dots, (a_{s1}, \dots, a_{ss})$ in
$\Lambda_{d_1, \dots, d_s}^q$. Thus 
\[ \left\{ \begin{matrix} 
g_1(w)^{a_{11}} \dots g_s(w)^{a_{1s}} = g_1(v)^{a_{11}} \dots g_s(v)^{a_{1s}},\\
g_1(w)^{a_{21}} \dots g_s(w)^{a_{2s}} = g_1(v)^{a_{21}} \dots g_s(v)^{a_{2s}},\\
\vdots \\
g_1(w)^{a_{s1}} \dots g_s(w)^{a_{ss}} = 
g_1(v)^{a_{s1}} \dots g_s(v)^{a_{ss}},  
\end{matrix} \right. \]
where the elements on the right hand side are non-zero.
We view $v$ as fixed and allow $w$ to range over the fiber of $\phi(v)$.
The matrix $A := (a_{ij})$ is invertible and $\det(A) \cdot  A^{-1}$ has
integer entries.  Thus, we can solve the above system
for $g_1^{\det(A)}(w), \dots, g_s^{\det A}(w)$. 

In conclusion, as $w$ ranges over the fiber of $\phi_q(v)$, we see
that $g_{s+1}(w) = \dots = g_r(w) = 0$ (by Claim 2) and 
$g_{1}(w) = \dots = g_s(w)$ assume only finitely many values.
Thus $w$ can only lie in finitely many $G$-orbits, as desired. 
\end{proof}

\begin{remark} \label{rem.noether} 
E.~Noether showed that $k[V]^G$ is generated by polynomials of degree
$\leqslant |G|$ as a $k$-algebra under the assumption that $\Char(k) = 0$.
The more general variant of the Noether bound used in the proof of
Proposition~\ref{prop.invariant3} (where $\Char(k) > 0$ is allowed, 
as long as $\Char(k)$ does not divide $|G|$)
is due to P.~Fleischmann, J.~Fogarty, and D.~Benson. For details and 
further references, see~\cite[Section 3.8]{derksen-kemper}.
\end{remark}

\section{Proof of Theorem~\ref{thm.horizontal2}(a)}

Set $d := \dim(B) = \max_{p} \, a(p)$.
Our goal is to construct a $d$-dimensional irreducible faithful
$G$-variety $X_d$ which is $p$-versal for every prime $p$.
This would imply $\pmed(G) \leqslant \dim(X_d) = d$, 
as desired.  

If $|G| = 1$ (or, equivalently, $d = 0$), we can 
take $X_d$ to be a point. Thus, from now on, we will 
assume that $G$ is non-trivial or, equivalently, $d \ge 1$.

Choose a sufficiently large prime integer $q$ so that 
$q \ne \Char(k)$, and every part of Lemma~\ref{prop.invariant3} 
holds; in particular, we will assume $q > |G|$.
This prime will remain fixed throughout the proof.
For notational simplicity we will write
$\psi \colon  \bbP(W) \dasharrow Y \subset \bbP^N$ for the rational 
map given by the linear system $k[W]_q^G$ of $G$-invariant
homogeneous polynomials of degree $q$, instead of $\psi_q$.
By part (a) of Lemma~\ref{prop.invariant3},
$\psi$ is regular away from $B$, and by part (b), $\psi$
is generically a $G$-torsor.  

Let $Y_n$ be a dense open 
subset of $Y$ such that $\psi \colon X_n \to Y_n$ is a $G$-torsor 
(and in particular, \'etale). Here $X_n$ is the preimage of $Y_n$ in
$\bbP(W \setminus B)$. The subscript $n$ in $X_n$ and $Y_n$ is intended
to remind us that $\dim(X_n)= \dim(Y_n) = n$, where
$n = \dim(V)= \dim(\bbP(W))$, as before. 
The idea of our construction of $X_d$ is to start with a $G$-invariant 
open subset $X_n$ of $\bbP(W \setminus B)$ and to construct successive
hyperplane sections $X_{n-1}, \dots, X_d$ recursively by appealing to 
Bertini's Theorem~\ref{thm.bertini}.

If $n = d$ then we are done. Indeed, our variety $X_n$ is $G$-equivariantly
birationally isomorphic to a vector space $V$, with a faithful linear
$G$-action. Hence, $X_n$ is versal, and, in particular, $p$-versal 
for every prime $p$.  Therefore, we may assume without loss of generality
that $n \geqslant d + 1 \geqslant 2$.

Since $X_n$ is birationally isomorphic to $V$, there exists an $F$-point
$x \in X_n(F)$, where $F/k$ is a finite separable field extension 
of degree prime to $q$. In fact, such points are dense in $X_n$.
Note that if $k$ is infinite, we can take $F = k$. 

By Theorem~\ref{thm.bertini} for sufficiently large $s_1$
there is a homogeneous polynomial $f \in k[y_0, \dots, y_N]$ 
of degree $q^{s_1}$ such that 

\smallskip
(i) $\Yfnote{(X_n)}{f_1}$ is geometrically irreducible, 

\smallskip
(ii) $\Yfnote{(Y_n)}{f_1}$ is smooth,

\smallskip
(iii) $\psi(x) \in \Yfnote{(Y_n)}{f_1}$,

\smallskip
(iv) $\dim(\Yfnote{(X_n)}{f_1}) = \dim(X_n)-1$.

\smallskip
\noindent
Here $y_0, \dots, y_N$ denote homogeneous coordinates on $\bbP^N$.

We now set $X_{n-1}:= \Yfnote{(X_n)}{f_1}$, $Y_{n-1}:= \Yfnote{(Y_n)}{f_1}$ and proceed 
to construct $Y_{n-2}, \dots, Y_{n-d}$ and $X_{n-2}, \dots, X_d$
recursively, where each $X_{n-i}$ is the preimage of $Y_{n-i}$ in
$\bbP(W \setminus B)$ under $\psi$, each $X_{n-i}$ is irreducible,
each $Y_{n-i}$ (and hence, $X_{n-i}$) is smooth of dimension $n-i$, 
each $Y_{n-i}$ contains $\psi(x)$, and each $Y_{n-i-1}$ is obtained
by intersecting $Y_{n-i}$ with a hypersurface $f_{i} = 0$ in $\bbP^N$, 
for a homogeneous polynomial $f_i \in k[y_0, \dots, y_N]$ of degree $q^{s_i}$. 

Note that since $\psi$ is given by the linear system of $k[V]_q^G$
of homogeneous $G$-invariant polynomials of degree $q$, 
$f_i$ lifts to a homogeneous polynomial $\psi^*(f_i)$
of degree $q^{s_1 + 1}$ on $\bbP(W)$. In other words,
\begin{equation} \label{e.hypersurfaces}
 X_d = (H[1] \cap \dots \cap H[n-d]) \setminus (\bbP(B) \cup 
\psi^{-1}(\overline{Y}_d \setminus Y_d) ) , 
\end{equation}
where $\overline{Y}_d$ is the closure of $Y_d$ in $\bbP^N$ and
$H[i]$ is a hypersurface of degree $q^{s_i + 1}$ in $\bbP(W)$
cut out by $\psi^*(f_i)$.

\smallskip
Since each $\psi \colon X_{n-i} \to Y_{n-i}$ is a $G$-torsor,
the $G$-action on $X_d$ is faithful. Thus
it remains to show that the $G$-action on $X_d$ is $p$-versal 
for every prime $p$. 

\smallskip
Case 1: $p = q$. Recall that the $G$-action on $X_d$ is $p$-versal 
if and only the $G_p$-action on $X_d$ is $p$-versal, where $G_p$ is 
a Sylow $p$-subgroup of $G$; see~\cite[Corollary 8.6]{versal}.
Since $q > |G|$, we have $G_q = \{ 1 \}$.
Thus in order to show that $X_d$ is $q$-versal it suffices to show that
$X_d$ has a $0$-cycle of degree prime to $q$; 
see~\cite[Lemma 8.2 and Theorem 8.3]{versal}. By our construction
$Y_d$ contains $\psi(x)$ and hence, $X_d$ contains $x$, where
$x$ is a point of degree prime to $q$. This shows that $X_d$ is $q$-versal. 

\smallskip
Case 2: $p \ne q$.  To show that the $G$-action on $X_d$ is $p$-versal 
it suffices to prove that for every field extension $K/k$,
with $K$ infinite, and every
$G$-torsor $T \to \Spec(K)$, the twisted $K$-variety 
$\, ^T X_d $ contains a $0$-cycle $Z$, whose degree over $K$ is a power of $q$
(and thus prime to $p$); see~\cite[Section 8]{versal}. 

Since the $G$-action on $\bbP(W)$ lifts to
a linear $G$-action on $W$, Hilbert's Theorem 90 tells us that
$\, ^T \bbP(W) = \bbP(W_K)$ is a projective space over $K$; see, e.g.,
\cite[Lemma 10.1]{versal}.
Twisting both sides of~\eqref{e.hypersurfaces} by $T$, we obtain
\[ ^T X_d = (\, ^T H[1] \cap \dots \cap { }^T H[n-d]) \setminus
(\, ^T \bbP(B) \cup \, { } ^T \psi^{-1}(\overline{Y}_d \setminus Y_d) )  \] 
in $\bbP(W_K)$.  We will construct the desired zero cycle $Z$ on
${ }^T X_d$ by intersecting ${ }^T X_d$ with 
$d$ hyperplanes $M_1, \dots, M_d$ in $\bbP(W_K)$
in general position. Note that since $Y_d$ is irreducible,
Lemma~\ref{prop.invariant3}(c) tells us that
\[ \dim_k \psi^{-1}(\overline{Y}_d \setminus Y_d) \leqslant 
\dim_k (\overline{Y}_d \setminus Y_d) \leqslant
\dim_k (Y_d) - 1 = d - 1  \, . \]
Since $\dim_k (\bbP(B)) = \dim_k (B) - 1 = d-1$, we see that
a linear subspace $M = M_1 \cap \dots \cap M_d$ 
of codimension $d$ in $\bbP(W_K)$
in general position misses both ${}^T \bbP(B)$ and 
${ }^T \psi^{-1}(\overline{Y}_d \setminus Y_d)$.

Let $Z$ be the intersection cycle obtained by intersecting $\, ^T X_d$ 
with $M$.  By~\cite[Lemma 10.1(c)]{versal}, 
each ${}^T H[i]$ is a hypersurface of degree $q^{s_i + 1}$
in $\bbP(W_K)$.
Hence, by Bezout's theorem~\cite[Proposition 8.4]{fulton},
\begin{align*}
\deg_K(Z) = & \deg(\, ^T H[1]) \cdot \ldots \cdot \deg(\, ^T H[n-d]) 
\cdot \deg(M_1) \cdot \ldots \cdot \deg(M_d) \\ 
   &= q^{s_1 + 1} \cdot  \ldots \cdot q^{s_{n-d} + 1} \cdot  
\underbrace{1 \cdot \ldots \cdot 1}_{\text{$d$ times}}
\end{align*}
is a power of $q$, as desired.
\qed

\section{$A$-groups} \label{sec:Agroups}

Let $G$ be a finite group, $p$ be a prime and $G_p$ be a Sylow $p$-subgroup
of $G$. Recall that $G$ is called an \emph{$A$-group} if $G_p$ is abelian 
for every $p$; see, e.g., \cite{ito,walter,broshi}. For the rest of this
section, with the exception of Conjecture~\ref{conj.elephant} below,
we will assume that the base field $k$ is of characteristic zero 
and $\zeta_e \in k$, where $e$ is the exponent of $G$.  

\begin{prop} \label{prop.Sylow-abelian}
Let $G$ be an $A$-group.  Then
\[ \pmed(G) = \max_p \, \ed(G; p) = \max_p \, \rank(G_p)  \] 
where the maximum is taken over all primes $p$.
\end{prop}

Here, as usual, by the rank of a finite abelian group $H$ we mean
the minimal number of generators of $H$.  

\begin{proof} The second equality is well known;
see, e.g.,~\cite[Corollary 7.3]{ry0}. Note also that this is a very 
special case of~\eqref{e.km}. In view of~\eqref{e.obvious},
in order to prove the first equality, we only need 
to show that $\pmed(G) \leqslant \max_p \, \rank(G_p)$. 

Let $p_1, \dots, p_r$ be the prime divisors of $|G|$ and
$d = \max \, \rank(G_{p_i})$, as $i$ ranges from $1$ to $r$.  
By \cite[Theorem 8.6]{ry} there exists a faithful
primitive $d$-dimensional $G$-variety $Y$ with smooth $k$-points 
$y_1, \dots, y_r$ such that $G_{p_i} \subset \Stab_G(y_i)$ 
for $i = 1, \dots, r$. 

Recall that ``primitive" means that $G$ transitively permutes 
the irreducible components of $Y_{\bar{k}}$.  We claim that any such
$Y$ is, in fact, absolutely irreducible. Let us assume this claim for 
a moment. 
The $G$-orbit of $y_i$ is a zero cycle of degree prime to $p_i$.
Thus for any given prime $p$, the degree of one of these orbits is prime 
to $p$.  By~\cite[Corollary 8.6(b)]{versal}, this implies that
$Y$ is $p$-versal for every $p$. Hence, $\pmed(G) \le \dim(Y) = d$, 
and the proposition follows.

It remains to show that $Y$ is absolutely irreducible.
After replacing $k$ by its algebraic closure $\bar{k}$, we may
assume that $k$ is algebraically closed.  Let $Y_0$ be an irreducible 
component of $Y$ and $H$ be the stabilizer of $Y_0$ in $G$.
Our goal is to prove that $H = G$. Since $G$ acts transitively on 
the irreducible components of $Y$, this will imply that $Y = Y_0$.

Since $y_i$ is a smooth point of $Y$, it lies on exactly one
irreducible component of $Y$, say on $g_i(Y_0)$ for some $g_i \in G$.
Since $y_i$ is $G_{p_i}$-invariant, $y_i$ also lies on $g g_i(Y_0)$ for every
$g \in G_{p_i}$. In other words, $g g_i (Y_0) = g_i(Y_0)$ for every 
$g \in G_{p_i}$
or equivalently, $g_i^{-1} G_{p_i} g_i \subset H$ for every $i =  1, \dots, s$.
This shows that $H$ contains a Sylow $p_i$-subgroup of $G$ for 
$i = 1, \dots, r$. Hence, $|H|$ is divisible by 
$|G_{p_i}|$ for every $i = 1, \dots, r$. We conclude that
$|H|$ is divisible by $|G| = |G_{p_1}| \cdot \dots \cdot |G_{p_s}|$
and hence, $H = G$.
\end{proof}

\begin{remark} \label{rem.base-field2}
The above argument relies, in a key way, on~\cite[Theorem 8.6]{ry}. 
This theorem is proved in~\cite{ry} over an algebraically closed 
field of characteristic $0$ but the proof goes through for any
$k$ as above. The condition that $\zeta_e \in k$, is necessary; 
it is not mentioned in~\cite[Remark 9.9]{ry} due to an oversight.
\end{remark}

\begin{example} \label{ex.pmed1}
If $G$ is a non-abelian group of order $pq$, where $p$ and $q$ 
are odd primes. Then Proposition~\ref{prop.Sylow-abelian}
tells us that $\pmed(G) = 1$. On the other hand, $\ed(G) \geqslant 2$;
see~\cite[Theorem 6.2]{bur}. This is, perhaps, the simplest 
example where $\pmed(G) < \ed(G)$.
\end{example}

\begin{remark}
Non-abelian simple $A$-groups  are classified in \cite[Theorem 3.2]{broshi}:
they are $J_1$, the first Janko group, and $\PSL_2(q)$ for $q > 3$ and
$q \equiv 0$, $3$, or $5 \pmod{8}$.
By Proposition~\ref{prop.Sylow-abelian},
\[ \pmed(G) = \begin{cases} \text{$3$, if $G \simeq J_1$,} \\
\text{$2$, if $G \simeq  \PSL_2(q)$, with $q$ as above.}
\end{cases} \]
On the other hand, by~\cite{beauville}, $\ed(G) \geqslant 4$ for 
any of these groups,
except for $G \simeq \PSL_2(5)$ and (possibly) $\PSL_2(11)$.
\end{remark}

It is natural to conjecture the following generalization 
of~\cite[Theorem 8.6]{ry}. 

\begin{conj} \label{conj.elephant}
Let $d$ be a positive integer.
Suppose $G$ is a finite group with subgroups $H_1, \dots, H_r$
such that $\rdim_k(H_i) \le d$ for all $i = 1, \ldots, r$.
Then there exists a $d$-dimensional $k$-variety $X$ 
with a faithful $G$-action and smooth $k$-points $x_1, \dots, x_r \in X$
such that $H_i$ fixes $x_i$ for each $i = 1, \dots, r$.
\end{conj}

Note that each $H_i$ must act faithfully on the tangent space of the
corresponding $x_i$ and so the condition that the representation dimension
of each $H_i$ should be $\le d$ is necessary.

Of particular interest is the special case where
$p_1, \dots, p_r$ are the distinct primes dividing $|G|$,
each $H_i$ is a Sylow $p_i$-subgroup, and $d$ is the maximum of
$\ed_k(G; p_i) = \rdim_k(H_i)$.
If Conjecture~\ref{conj.elephant}
could be established in this special case, then 
the argument we used in the proof of Proposition~\ref{prop.Sylow-abelian}
would show that the $G$-action on $X$ is $p$-versal for every prime $p$  
and, consequently, that~\eqref{e.elephant} holds for $G$.
We have not been able to prove~\eqref{e.elephant} by this method beyond 
the case of $A$-groups.

\section{Examples}
In this section we present two examples that complement
Theorem~\ref{thm.horizontal2}(b). 
Example~\ref{ex.Alt_n} shows that the inequality of 
Theorem~\ref{thm.horizontal2}(a) is in fact
an equality, for the natural $n$-dimensional
representation $V$ of the alternating group $\Alt_n$. Note that  
Theorem~\ref{thm.horizontal2}(b) cannot be applied to 
$\Alt_n \subset \GL(V)$, since $\Alt_n$ contains no pseudo-reflections. 
Nevertheless, the conclusion of Theorem~\ref{thm.horizontal2}(b) 
continues to hold 
in this case. On the other hand, Example~\ref{ex.semidirect}
shows that for $G = \bbZ/5\bbZ \rtimes \bbZ/4\bbZ$ the inequality 
of Theorem~\ref{thm.horizontal2}(a) is strict 
for every faithful representation $G \hookrightarrow \GL(V)$. 

\begin{example} \label{ex.Alt_n}
$\pmed(\Alt_n) = \ed(\Alt_n; 2) = 
2 \left\lfloor \dfrac{n}{4} \right\rfloor$ for any $n \geqslant 4$.
\end{example}

\begin{proof} 
Since $\Alt_n$ contains an elementary abelian subgroup of rank 
$2 \left\lfloor \dfrac{n}{4} \right\rfloor$ generated by
$(12)(34), (13)(24), (56)(78)$, etc., we have 
$\pmed(\Alt_n) \geqslant \ed(\Alt_n; 2) = 2 \left\lfloor 
\dfrac{n}{4} \right\rfloor$; see~\cite[Theorem 6.7(c)]{bur}.

We will now deduce the opposite inequality, 
\begin{equation} \label{e.alt}
\pmed(\Alt_n) \le 2 \left\lfloor \dfrac{n}{4} \right\rfloor
\end{equation}
from Theorem~\ref{thm.horizontal2}(a). Let $V = k^n$ 
be the natural representation of $\Sym_n$.
One checks that for any $g \in \Sym_n$ and any prime $p$,
the dimension of the eigenspace $V(g, \zeta_p)$ is the number 
of cycles of length divisible by $p$ in the cycle decomposition of $g$.
Thus 
\[ a(p) = \max_{g \in \Alt_n} \,  \dim \, V(g, \zeta_p) = 
\begin{cases} 
\text{$\left\lfloor n/p \right\rfloor$, if $p$ is odd, and} \\
\text{$2 \left\lfloor n/4 \right\rfloor$, if $p = 2$,}
\end{cases} \]
Since  we are assuming that $n \ge 4$,
the maximal value of $a(p)$ is attained at $p = 2$. The inequality
\eqref{e.alt} now follows from Theorem~\ref{thm.horizontal2}(a), as desired. 
\end{proof}

\begin{example} \label{ex.semidirect}
Let $G = \bbZ/5\bbZ \rtimes \bbZ/4\bbZ$, where $\bbZ/4\bbZ$ 
acts faithfully on $\bbZ/5\bbZ$. Assume $\zeta_{20} \in k$. Then

\smallskip
(a) $\pmed(G) = 1$, but

\smallskip
(b) $a_{\phi}(2) \geqslant 2$ for every faithful 
representation $\phi \colon G \hookrightarrow \GL(V)$.
\end{example}

\begin{proof} Since the Sylow subgroups of $G$ are
$\bbZ/5\bbZ$ and $\bbZ/4\bbZ$, part (a) follows from
Proposition~\ref{prop.Sylow-abelian}.

(b) Each of the four characters $\bbZ/4 \bbZ \to k^*$ 
induces a $1$-dimensional representation $G \to \GL_1$.
We will denote these representations by $\phi_0 = \id$,
$\phi_1$, $\phi_2$, and $\phi_3$. Let $\phi_4 = \Ind_{\bbZ/5 \bbZ}^G(\chi)$,
where $\chi$ is a non-trivial multiplicative character $\bbZ/5 \bbZ \to k^*$.
We see that $\phi_4$ is a faithful irreducible $4$-dimensional 
representation of $G$ (irreducibility follows, e.g, from Mackey's criterion) 
and $a_{\phi_4}(2) = 2$.  Since 
$\dim(\phi_0)^2 + \dots + \dim(\phi_4)^2 = 4 \cdot 1^2 + 4^2 = 20 = |G|$,
$\phi_0, \dots, \phi_4$ are the only irreducible representations of $G$. 
Moreover, since $\bbZ/5 \bbZ$
lies in the kernel of $\phi_0, \dots, \phi_3$, every faithful representation
$\phi \colon G \hookrightarrow \GL(V)$ must contain a copy of $\phi_4$.
Thus $a_{\phi}(2) \geqslant a_{\phi_4}(2) = 2$.
\end{proof}

\begin{remark} A.~Ledet showed that
$\ed(\bbZ/5\bbZ \rtimes \bbZ/4\bbZ) = 2$; see~\cite[p.~426]{ledet}.
Note that in~\cite{ledet} this group is denoted by $C_5$.
\end{remark}

\section*{acknowledgements}
The authors are grateful to I.~Dolgachev 
for suggesting the geometric construction used in the proof of
Lemma~\ref{lem:WE6} and B.~Poonen for a helpful discussion of
Bertini's theorem over finite fields and for sending us 
a draft version of his preprint~\cite{CPbertini}.
We would also like to thank 
M.~Garc{\'i}a-Armas,
G.~I.~Lehrer, 
R.~L\"otscher,
M.~MacDonald,
J.-P.~Serre,
and an anonymous referee
for helpful comments.

\end{document}